\documentclass{amsart}
\usepackage{graphics, verbatim, amsmath, amssymb, amsthm, amsfonts,% epstopdf, psfrag,
mathtools,hyperref}	
\usepackage{ifthen}
\newtheorem{proposition}{Proposition}[section]
\newtheorem{theorem}[proposition]{Theorem}

\newtheorem{lemma}[proposition]{Lemma}
\newtheorem{prop}[proposition]{Proposition}
\newtheorem{cor}[proposition]{Corollary}

\theoremstyle{definition}
\newtheorem{example}[proposition]{Example}

\theoremstyle{remark}
\newtheorem{remark}[proposition]{Remark}

\numberwithin{equation}{section}

\usepackage{color}

\newcommand{\newword}[1]{\textbf{\emph{#1}}}

\newcommand{\integers}{\mathbb Z}
\newcommand{\reals}{\mathbb R}

\newcommand{\ep}{\varepsilon}

\newcommand{\Seed}{\operatorname{Seed}}
\newcommand{\Camb}{\operatorname{Camb}}

\newcommand{\sgn}{\operatorname{sgn}}
\renewcommand{\int}{\operatorname{int}}

\newcommand{\cov}{\mathrm{cov}}

\newcommand{\covered}{{\,\,<\!\!\!\!\cdot\,\,\,}}
\newcommand{\set}[1]{{\lbrace #1 \rbrace}}

\newcommand{\pidown}{\pi_\downarrow}

\newcommand{\br}[1]{{\langle #1 \rangle}}
\newcommand{\g}{{\mathbf g}}

\renewcommand{\c}{{\mathbf c}}
\renewcommand{\b}{{\mathbf b}}

\newcommand{\A}{{\mathcal A}}

\newcommand{\tB}{{\widetilde{B}}}

\newcommand{\F}{{\mathcal F}}

\newcommand{\join}{\vee}

\renewcommand{\Join}{\bigvee}
\newcommand{\Meet}{\bigwedge}

\newcommand{\ck}{^{\vee\!}}

\newcommand{\dashname}[1]{\stackrel{#1}{\begin{picture}(22,3)\put(0,2.5){\line(1,0){22}}\end{picture}}}

\renewcommand{\th}{^\mathrm{th}}

\newcommand{\1}{{\hat{1}}}

\newcommand{\FF}{\mathbb{F}}
\newcommand{\ZZ}{\mathbb{Z}}

\newcommand{\RR}{\mathbb{R}}

\newcommand{\Tits}{\mathrm{Tits}}
\newcommand{\Cone}{\mathrm{Cone}}

\newcommand{\Proj}{\mathrm{Proj}}
\newcommand{\relint}{\mathrm{relint}}

\newcommand{\Ex}{\mathrm{Ex}}

\DeclareMathOperator{\Span}{Span}

\DeclareMathOperator{\inv}{inv}

\DeclareMathOperator{\Cart}{Cart}
\newcommand{\twomatrix}[2]{\begin{bsmallmatrix} #1 \\ #2 \end{bsmallmatrix} }
\newcommand{\ab}{\uparrow}
\newcommand{\bel}{\downarrow}

\newcommand{\DF}{{\mathcal {DF}}}

\newcommand{\aff}{\mathrm{aff}}
\newcommand{\DCamb}{\operatorname{DCamb}}
\newcommand{\DC}{\mathrm{DC}}

\begin{document}

\title[Affine Cambrian frameworks]{Cambrian frameworks for cluster algebras of affine type}
\author{Nathan Reading and David E Speyer}

\thanks{Nathan Reading was partially supported by NSA grant H98230-09-1-0056, by Simons Foundation grant \#209288, and by NSF grant DMS-1101568.   David E Speyer was supported in part by a Clay Research Fellowship}

\begin{abstract}
We give a combinatorial model for the exchange graph and $\g$-vector fan associated to any acyclic exchange matrix $B$ of affine type. 
More specifically, we construct a reflection framework for $B$ in the sense of [N. Reading and D. E. Speyer, ``Combinatorial frameworks for cluster algebras''] and establish good properties of this framework.
The framework (and in particular the $\g$-vector fan) is constructed by combining a copy of the Cambrian fan for $B$ with an antipodal copy of the Cambrian fan for $-B$.
\end{abstract}

\maketitle

\setcounter{tocdepth}{2}
\tableofcontents

\section{Introduction}\label{intro sec} 
In a series of papers \cite{cambrian,sortable,camb_fan}, the authors studied cluster algebras of finite type in terms of the combinatorics and geometry of finite Coxeter groups, and in particular sortable elements and Cambrian lattices.
In \cite{typefree,cyclic,framework}, 
these constructions were extended to infinite Coxeter groups.
Sortable elements and Cambrian semilattices were shown in \cite{framework} to produce combinatorial models of cluster algebras of infinite type, with an important limitation:  
Sortable elements can only model the part of the exchange graph that corresponds to clusters whose $\g$-vector cones intersect the interior of the Tits cone.

In this paper, we show how to extend the sortable/Cambrian setup to obtain a complete combinatorial model when
$B$ is acyclic and its Cartan companion $\Cart(B)$
 is of affine type. The basic idea can be illustrated by a very simple example.   

\begin{figure}[ht]
\includegraphics{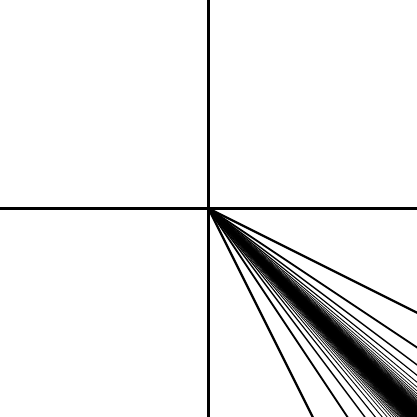}
\caption{The $\g$-vector fan for $B$ with $\Cart(B)$ of type $\widetilde A_1$}
\label{A1tilde g}
\end{figure}
\begin{figure}[ht]
\includegraphics{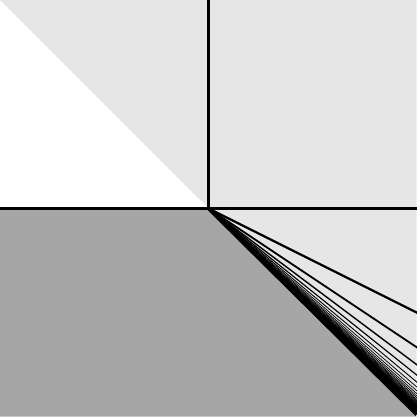}\qquad\qquad\includegraphics{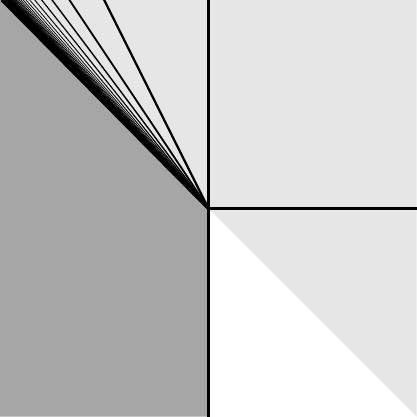}
\caption{The Cambrian fans for $B$ and for $-B$}
\label{A1tilde Camb}
\end{figure}
Suppose $B=\begin{bsmallmatrix*}[r]0& 2 \\ -2 & 0 \end{bsmallmatrix*}$.
Its Cartan companion $\Cart(B)=\begin{bsmallmatrix*}[r] 2 & -2 \\ -2 & 2 \end{bsmallmatrix*}$ is of type $\tilde{A}_1$.
Figure~\ref{A1tilde g} shows the $\g$-vector fan associated to $B$.
The left picture of Figure~\ref{A1tilde Camb} shows the Cambrian fan associated to $B$, while the right picture shows the Cambrian fan associated to $-B$.
In each Cambrian fan picture, the Tits cone is identified by light gray shading and the area outside the Cambrian fan is identified by dark gray shading.
Two surprising things happen.
First, the Cambrian fan for $B$ is ``compatible'' with the image under the antipodal map of the Cambrian fan for $-B$, in the sense that the union of these two fans is again a fan (the \newword{doubled Cambrian fan}).
Second, the doubled Cambrian fan coincides with the $\g$-vector fan, so that the dual graph to the doubled Cambrian fan is isomorphic to the exchange graph.

The main results of the paper are that the first of these surprising things happens for all acyclic $B$ and that the second happens whenever $B$ is acyclic and $\Cart(B)$ 
 is of affine type.  
For richer examples than the rank-$2$ example described above, see Examples~\ref{affineG2 ex}, \ref{affineG2 frame ex}, \ref{344 ex} and~\ref{affineG2 slice ex}.

We now state our main results more formally.
For each acyclic exchange matrix $B$, we construct a doubled Cambrian fan $\DF_c$ as the union of the Cambrian fan for $B$ with the antipodal Cambrian fan for $-B$.
The \newword{doubled Cambrian framework} is the pair $(\DCamb_c, \DC_c)$, 
where $\DCamb_c$ is the dual graph to $\DF_c$ and $\DC_c$ is a certain labeling of $\DCamb_c$ by roots.  
In the language of \cite{framework}, the key result is the following.

\begin{theorem}\label{AffineDoubleFramework}
Suppose that $B$ is acyclic and $\Cart(B)$ 
is of affine type. Then $(\DCamb_c, \DC_c)$ is a complete, exact, well-connected, polyhedral, simply connected reflection framework.
\end{theorem}

Let $\A_\bullet(B)$ be the principal-coefficients cluster algebra associated to $B$.
Using results of \cite{framework} (as we explain in Section~\ref{frame and clus sec}), the following corollaries are easily obtained.  

\begin{cor}\label{Affine exch}
If $B$ is acyclic and $\Cart(B)$ 
is of affine type, then the exchange graph of the cluster algebra $\A_\bullet(B)$ is isomorphic to $\DCamb_c$. 
The isomorphism sends a vertex $v$ of $\DCamb_c$ to a seed whose $\c$-vectors are (simple-root coordinates of) $\DC_c(v)$.
\end{cor}

\begin{cor}\label{Affine g}
If $B$ is acyclic and $\Cart(B)$
is of affine type, then  $\DF_c$ coincides with the fan of $\g$-vector cones for the cluster algebra $\A_\bullet(B)$.
\end{cor}
Corollary~\ref{Affine g} refers to the $\g$-vector cones defined by interpreting $\g$-vectors as fundamental-weight coordinates of vectors in the weight lattice.  
See \cite[Remark~5.17]{framework} and Section~\ref{ca background sec} of the present paper.  
When $B$ is of infinite non-affine type, the doubled Cambrian fan may be a proper subfan of the $\g$-vector fan.
(See Example~\ref{344 ex}.)
However, the doubled Cambrian fan coincides with the $\g$-vector fan for all $2\times 2$ exchange matrices, even those of non-affine infinite type.  
(See Remark~\ref{rank 2}.)  

One of the initial motivations for this work was to prove the affine-type case of many of the standard structural conjectures on cluster algebras.
Indeed, theorems about frameworks from \cite{framework} combined with the results stated above accomplish that task.
(See Corollary~\ref{affine type conj}.)
Many of the standard conjectures have been turned into theorems in full generality (not merely the affine case) using the machinery of scattering diagrams \cite{GHKK}.
See also the table at the end of \cite[Section~3.3]{framework} for more on the previous status of these conjectures.

The approach to cluster algebras via scattering diagrams makes it \emph{more} important, not less important, to make explicit constructions in key special cases, like the affine case.
Scattering diagrams consist of a rational fan, decorated with certain rational functions, and containing the $\g$-vector fan as a subfan.
Other constructions, notably semi-invariants (see~\cite{IPT} for a detailed treatment of the affine case,~\cite{IOTW1} for the acyclic case and~\cite{IOTW2} for the beginnings of an investigation of the general case)
and the mutation fan~\cite{universal}  also yield fans that contain the $\g$-vector fan.

However, none of these methods give a direct combinatorial description of the fan, and it is difficult to use them to study the regions of the fan.
Scattering diagrams natively construct the walls of the fan and it is difficult to use them to discuss the chambers cut out by those walls. 
(Not impossible though! See~\cite{Muller} for some progress along these lines.) 
Semi-invariants also naturally construct the walls, although connections have been found to cluster tilting and $\tau$-tilting modules which correspond to the regions~\cite{IOTW1, IT}. 
Neither method naturally connects the fan directly to lattice theory, although torsion classes form a lattice which has been related to semi-invariants~\cite{IT}.

In contrast, our methods provide direct descriptions of the $\g$-vector fan in the affine types in terms of the combinatorics of affine Coxeter groups and root systems.
We label the regions of the $\g$-vector fan by sortable elements in those groups, which can be described either in terms of  reduced words or in terms of pattern avoidance~(\cite{typefree}).
And our methods are built on lattice theoretic and geometric properties which are well suited for proving global properties of the fans we study.

The remainder of the paper is devoted to constructing the doubled Cambrian framework and fan and proving Theorem~\ref{AffineDoubleFramework}.
We begin with background in Section~\ref{background sec}.
We define the doubled Cambrian framework and fan for any acyclic exchange matrix in Section~\ref{double sec}, where we also prove the part of Theorem~\ref{AffineDoubleFramework} that does not need the hypothesis the $\Cart(B)$ is of affine type.
The proof of Theorem~\ref{AffineDoubleFramework} is completed in Section~\ref{affine sec}.

Most of the affine Dynkin diagrams are trees, and thus admit no non-acyclic orientations. 
The exception is $\tilde A_{n-1}$ ($n \geq 3$), which is an $n$-cycle. 
Thus, in almost every case when $\Cart(B)$ is of affine type, $B$ is acyclic and the doubled Cambrian framework construction provides a complete framework for $B$.
However, the cyclic orientation in type $\tilde A_{n-1}$ ($n \geq 3$) does not fit into the doubled Cambrian framework.
In this case, the associated cluster algebra is of finite type $D_n$.
In~\cite{cyclica}, we construct a complete framework for the cyclically oriented $n$-cycle using the $\tilde{A}_{n-1}$ root system and a variation on the doubled Cambrian framework idea.

\section{Background}\label{background sec}  

\subsection{Reflection frameworks}\label{frame sec}
In this section, we review the background material on frameworks from~\cite{framework}. 
The doubled Cambrian frameworks belong to a special class of frameworks called \newword{reflection frameworks}.
For that reason, in this paper, we review only the definition of reflections frameworks, and not the general definition of frameworks that is required to model cluster algebras in general.

The starting point is an \newword{exchange matrix} $B=[b_{ij}]$, with rows and columns indexed by a set~$I$, with $|I|=n$.
This is a skew-symmetrizable integer matrix, meaning that there exists a positive real-valued function~$d$ on $I$ with $d(i) b_{ij}=-d(j) b_{ji}$ for all $i,j\in I$.
The \newword{Cartan companion} $\Cart(B)$ of $B$ is the square matrix $A=[a_{ij}]$ with diagonal entries $2$ and off-diagonal entries $a_{ij}=-|b_{ij}|$.  
This is a symmetrizable generalized Cartan matrix, and specifically $d(i) a_{ij}=d(j) a_{ji}$ for all $i,j\in I$.
(See \cite{Kac} for more on Cartan matrices.  Or, for an exposition tailored to the present purposes, see \cite[Section~2.2]{typefree}.)

We take $V$ to be a real vector space of dimension $n$ with a basis $\Pi=\set{\alpha_i:i\in I}$.
We write $V^*$ for the dual vector space to $V$ and write $\br{x,y}$ for the canonical pairing between $x\in V^*$ and $y\in V$.
The vectors $\alpha_i$ are the \newword{simple roots}. 
The vectors $\alpha_i\ck= d(i)^{-1} \alpha_i$ are the \newword{simple co-roots}.
The set of simple co-roots is written~$\Pi\ck$.
We define a bilinear form $\omega$ 
whose matrix is $B$, written in terms of the simple co-root basis on one side and the simple root basis on the other.
Specifically, we set $\omega(\alpha_{i}\ck,\alpha_{j})=b_{ij}$.  
It is easily checked that $\omega$ is skew-symmetric.

The Cartan matrix $A$ defines a Coxeter group $W$ and a specific reflection representation of $W$ on $V$.
The group is generated by the set $S=\set{s_i:i\in I}$ of \newword{simple reflections}, where each $s_i$ acts on a simple root $\alpha_j$ by $s_i(\alpha_j)=\alpha_j-a_{ij} \alpha_i$.
It follows that the action on a simple co-root $\alpha\ck_j$ is $s_i(\alpha\ck_j)=\alpha\ck_j-a_{ji} \alpha\ck_i$.
Furthermore,~$A$ encodes a symmetric bilinear form $K$ on $V$ with $K(\alpha\ck_i, \alpha_j)=a_{ij}$.
The action of $W$ can be rewritten as $s_i(\alpha_j)=\alpha_j-K(\alpha\ck_i, \alpha_j) \alpha_i$, so that $s_i$ acts as a reflection with respect to the form $K$.
In particular, the action of $W$ preserves the form $K$.

When it is convenient, we will replace the indexing set $I$ with the set $S=\set{s_i:i\in I}$ of simple reflections in $W$.
For example, the simple root $\alpha_i$ may appear as $\alpha_s$ for $s=s_i$.

The \newword{real roots} are the vectors in the $W$-orbit of $\Pi$, and the \newword{real co-roots} are the vectors in the $W$-orbit of $\Pi\ck$.
(Imaginary roots make an appearance in Section~\ref{affine root sec}.)  
The set of all real roots constitutes the \newword{(real) root system} $\Phi$ associated to $A$.
The root system is a subset of the \newword{root lattice}, the lattice in $V$ generated by $\Pi$.
The root system is the disjoint union of \newword{positive roots} (roots in the nonnegative linear span of $\Pi$) and \newword{negative roots} (roots in the nonpositive linear span of $\Pi$).
Each root $\beta$ has a corresponding co-root $\beta\ck$, related by the scaling $\beta\ck=2\frac{\beta}{K(\beta,\beta)}$.
The set of all co-roots is a root system in its own right, associated to the Cartan matrix $A^T$.
The integer span of the simple co-roots is a lattice called the \newword{co-root lattice}.
A \newword{reflection} in $W$ is an element conjugate to an element of $S$.
For each root $\beta$, there is a reflection $t$ in $W$ such that $tx=x-K(\beta\ck, x) \beta$ for all $x\in V$.
This defines a bijection between positive roots and reflections in $W$.
We write $\beta_t$ for the positive root associated to a reflection $t$.

A \newword{quasi-graph} is a graph with ordinary or \newword{full} edges and also \newword{half-edges}.
Half-edges should be though of as edges that dangle from a vertex without connecting that vertex to any other.
More formally, a quasi-graph is a hypergraph with edges of size $1$ or $2$.
We will assume our quasi-graphs to be \newword{simple}: no two edges connect the same pair of vertices, and every full edge connects two distinct vertices.
We will also consider only quasi-graphs that are \newword{regular of degree $n$}, meaning that each vertex is incident to exactly $n$-edges (i.e.\ $k$ half-edges and $n-k$ full edges for some $k$ from $0$ to $n$).
Finally, we will consider \newword{connected} quasi-graphs, meaning that, deleting half-edges, the remaining graph is connected in the usual sense.
A pair $(v,e)$ consisting of a vertex $v$ and an edge $e$ containing $v$, is called an \newword{incident pair}.
Let $I(v)$ denote the set of edges $e$ containing the vertex~$v$.  

We will define a reflection framework to be a connected $n$-regular quasi-graph $G$ together with a labeling $C$ which labels each incident pair in $G$ by a vector $C(v,e)$ in $V$.
The set $\set{C(v,e):e\in I(v)}$ of labels on a vertex $v$ will be written $C(v)$.
The labeling must satisfy certain conditions that we now explain.\\

\noindent
\textbf{Base condition:}  
There exists a vertex $v_b$ of $G$ such that $C(v_b)$ is the set of simple roots of $\Phi$.\\

In light of the Base condition, we identify the indexing set $I$ with $I(v_b)$.
Specifically, we identify $e\in I(v_b)$ with the index $i\in I$ such that $C(v_b,e)=\alpha_i$.\\

\noindent
\textbf{Root condition:}  
Each label $C(v,e)$ is a real root in $\Phi$.\\

In particular, each label $C(v,e)$ has a corresponding \newword{co-label} $C\ck(v,e)$, which is the co-root corresponding to the root $C(v,e)$.
The set of co-labels on a vertex $v$ is denoted by $C\ck(v)$.

For the benefit of the reader who is reading this work together with~\cite{framework}, we recall that, in a general framework, a label $C(v,e)$ is not necessarily a root, so there is no reasonable notion of a ``co-root corresponding to $C(v,e)$.''
Instead, the co-labels are defined separately to obey a Co-transition condition that is dual to the Transition condition.
In a reflection framework, we can define co-labels more simply as described above.
See~\cite[Proposition 2.13]{framework}.

The \newword{Euler form} $E$ associated to $B$ is defined by:  
\[E(\alpha\ck_{i},\alpha_{j})=\left\lbrace\begin{array}{ll}
\min(b_{ij},0)   &\mbox{if } i\neq j,\mbox{ or}\\
1&\mbox{if }i=j.
\end{array}\right.\]
The symmetric bilinear form $K$ obeys $K(\alpha, \beta) = E(\alpha, \beta)+E(\beta, \alpha)$ for any $\alpha$, $\beta \in V$. 
The skew-symmetric bilinear form $\omega$ satisfies $\omega(\alpha, \beta) = E(\alpha, \beta) - E(\beta, \alpha)$. 

Supposing the Root condition holds for $(G,C)$, for each vertex $v$ of $G$ we define $C_+(v)$ to be the set of positive roots in $C(v)$ and $C_-(v)$ to be the set of negative roots in $C(v)$.
We also define a directed graph $\Gamma(v)$ whose vertex set is $C(v)$ and which has
an edge $\beta \to \beta'$ whenever $\beta\neq\beta'$ and
$E(\beta, \beta') \neq 0$. 
\\

\noindent
\textbf{Euler conditions:}
Let $v$ be a vertex of $G$ and let $e$ and $f$ be distinct edges incident to $v$. 
Write $\beta=C(v,e)$ and $\gamma=C(v,f)$.
Then  
\begin{enumerate}
\item[(E1) ] If $\beta\in C_+(v)$ and $\gamma\in C_-(v)$ then $E(\beta,\gamma)=0$. 
\item[(E2) ] If $\sgn(\beta)=\sgn(\gamma)$ then $E(\beta,\gamma)\le0$. 
\item[(E3) ] The graph $\Gamma(v)$ is acyclic.\\
\end{enumerate}
Here $\sgn(\beta)=1$ if $\beta$ is a positive root and $-1$ if $\beta$ is a negative root.

\noindent
\textbf{Reflection condition:}
Suppose $v$ and $v'$ are distinct vertices incident to the same edge $e$.
If $\beta=C(v,e)=\pm\beta_t$ for some reflection $t$ and $\gamma\in C(v)$, then $C(v')$ contains the root 
\[\gamma'=\left\lbrace\begin{array}{ll}
t\gamma&\mbox{if }\omega(\beta_t\ck,\gamma)\ge 0,\mbox{ or}\\
\gamma&\mbox{if }\omega(\beta_t\ck,\gamma)<0.\\
\end{array}\right.\]
\smallskip

A \newword{reflection framework} for $B$ is a pair $(G,C)$ consisting of a connected, $n$-regular quasi-graph $G$, with a labeling $C$ satisfying the Base condition, the Root condition, the Reflection condition, and the Euler conditions (E1), (E2), and (E3).
If a reflection framework exists for $B$, then in particular the graph $\Gamma(v_b)$ is acyclic.
This condition on $\Gamma(v_b)$ is equivalent to saying that we can reorder the rows and columns of $B$ so that $b_{ij}>0$ when $i<j$.
When such a reordering is possible, we say that $B$ is \newword{acyclic}.  
The more general notion of a \newword{framework} applies to arbitrary exchange matrices $B$ with no requirement of acyclicity.
See~\cite[Section~2]{framework} for details.

The following is \cite[Proposition~2.4]{framework}, specialized to reflection frameworks.  
\begin{prop}\label{basis}
Suppose $(G,C)$ is a reflection framework for $B$ and let $v$ be a vertex of $G$.
Then the label set $C(v)$ is a basis for the root lattice and the co-label set $C\ck(v)$ is a basis for the co-root lattice.
\end{prop}

In particular, no label occurs twice in $C(v)$.
Accordingly, the Reflection condition lets us relate the sets $I(v)$ and $I(v')$ of edges incident to adjacent vertices $v$ and $v'$.
Specifically, if an edge $e$ connects $v$ to $v'$, we define a function $\mu_e$ from $I(v)$ to $I(v')$.
We set $\mu_e(e)$ to be $e$, and for each $f\in I(v)\setminus\set{e}$, taking $\gamma=C(v,f)$, we define $\mu_e(f)$ to be the edge $f'\in I(v')$ such that $C(v',f')=\gamma'$ in the Reflection condition.
The following version of the Reflection condition is useful, although it amounts only to a restatement of the definition of $\mu_e$:
\\

\noindent
\textbf{Reflection condition, restated:}
Suppose $v$ and $v'$ are distinct vertices incident to the same edge $e$.
Let $C(v,e)=\pm\beta_t$ for some reflection $t$.
Let $f\in I(v)$ and write~$\gamma$ for $C(v,f)$.
Then 
\[C(v',\mu_e(f))=\left\lbrace\begin{array}{ll}
t\gamma&\mbox{if }\omega(\beta_t\ck,\gamma)\ge 0,\mbox{ or}\\
\gamma&\mbox{if }\omega(\beta_t\ck,\gamma)<0.
\end{array}\right.\]

The \newword{fundamental weights} associated to the root system $\Phi$ are the vectors in the basis of $V^*$ that is dual to $\Pi\ck$.
The notation $\rho_i$ stands for the fundamental weight that is dual to $\alpha\ck_i$ in this dual basis.
The lattice in $V^*$ generated by the fundamental weights is called the \newword{weight lattice}.
Proposition~\ref{basis} implies that the dual basis to $C\ck(v)$ is a basis for the weight lattice.
Write $R(v)$ for this dual basis and write $R(v,e)$ for the vector dual to $C\ck(v,e)$ in $R(v)$.
The following is \cite[Proposition~2.8]{framework}, specialized to reflection frameworks. 
\begin{prop}\label{dual adjacent}
Let $(G,C)$ be a reflection framework for $B$ and let $v$ and $v'$ be adjacent vertices of $G$.
Then $R(v)\cap R(v')$ contains exactly $n-1$ vectors.
Specifically, if $e$ is the edge connecting $v$ to $v'$ and $f\in I(v)\setminus\set{e}$, then $R(v,f)=R(v',\mu_e(f))$.
Also, $R(v,e)$ and $R(v', e)$ lie on opposite sides of the hyperplane spanned by $R(v) \cap R(v')$.
\end{prop}

\subsection{Polyhedral geometry in $V^*$}\label{poly sec}  
A \newword{closed polyhedral cone} (or simply a \newword{cone}) in $V^*$ is a set of the form ${\bigcap \{ x \in V^* : \br{x, \beta_i} \geq 0 \}}$ for vectors $\beta_1$, \ldots, $\beta_k$ in~$V$.  
A subset of $V^{\ast}$ is a closed polyhedral cone if and only if it is the nonnegative linear span of a finite set of vectors in $V^*$.
A \newword{simplicial cone} is a cone such that $\beta_1$, \ldots, $\beta_k$ form a basis for $V$.
Equivalently, a simplicial cone is the nonnegative linear span of a basis for $V^*$.

A \newword{face} of a cone $F$ is a subset $G$ of $F$ such that there is a linear functional on $V^*$ that is $0$ on $G$ and nonnegative on $F$.
Taking the functional to be zero, we see that $F$ is a face of itself.  
A face $G$ of $F$ with $\dim(G)=\dim(F)-1$ is called a \newword{facet} of $F$.
The \newword{relative interior} of a cone is the cone minus (set-theoretically) its proper faces.

Two cones \newword{meet nicely} if their intersection is a face of each of them.
A \newword{fan} is a set $\F$ of cones such that (1) if $F$ is in $\F$ then all faces of $F$ are in $\F$, and (2) any two cones $F_1$ and $F_2$ in $\mathcal{F}$ meet nicely.
The fan is \newword{simplicial} if its maximal cones are simplicial.
We will use some well-known easy facts:
Each face of a cone is itself a cone.
The relation ``is a face of'' is transitive.
Two faces of the same cone meet nicely.
A set $\F$ of cones satisfying (1) is a fan if and only if any two of its \emph{maximal} cones meet nicely.
References for these facts can be found in \cite[Section~3.2]{framework}.

\subsection{Cluster algebras of geometric type}\label{ca background sec}
As before, take $I$ to be an indexing set with $|I|=n$.
Let $J$ be an indexing set, disjoint from $I$, with $|J|=m$.
An \newword{extended exchange matrix} 
 is an $(n+m)\times n$ integer matrix with columns
indexed by $I$ and rows indexed by the disjoint union $I \sqcup J$, such that the $n\times n$ submatrix whose rows are indexed by $I$ is skew-symmetrizable.
This skew-symmetrizable submatrix is the \newword{underlying exchange matrix}.
Let $x_i$ be indeterminates indexed by $i$ in $I \sqcup J$, 
and let $\FF$ be the field of rational functions in these indeterminates, with rational coefficients. 
A \newword{cluster} is an $n$-tuple of algebraically independent elements of $\FF$.
The individual elements of the cluster are called \newword{cluster variables}.
A \newword{seed} is a pair consisting of an extended exchange matrix and a cluster.

Let $T$ be the $n$-regular tree.
For each edge $e$ of $T$, connecting vertices $v$ and $v'$, fix a pair of inverse bijections between the set $I(v)$ of edges incident to $v$ and the set $I(v')$.
We will call both maps $\mu_e$ and let the context distinguish the two.
Now fix an \newword{initial seed} $(\tB,X)$ where $\tB$ is an extended exchange matrix with underlying exchange matrix $B$ and $X$ is a cluster, which we may as well take to be $x_i:i\in I$.
We will  
associate to each vertex $v$ of $T$ a seed $(\tB^v,X^v)$.
In each seed $(\tB^v,X^v)$, the rows of $\tB^v$ are indexed by $I(v)\sqcup J$,
while the columns of $\tB$, as well as the cluster variables, are indexed by $I(v)$.
Choose some vertex $v_b$ of $T$, identify $I$ with $I(v_b)$, and take $(\tB^{v_b},X^{v_b})$ to be the initial seed $(\tB,X)$.
The remaining seeds are defined recursively by \newword{seed mutations}.
Let $e$ be an edge connecting 
$v$ to $v'$.   
We extend the map $\mu_e$ to a bijection from $I(v)\sqcup J$ to $I(v')\sqcup J$ by letting the map fix $J$
 pointwise.  \\

\noindent
\textbf{Matrix mutation.} 
The matrices $\tB^v$ and $\tB^{v'}$ are related by
\begin{equation} 
b^{v'}_{\mu_e(p)\mu_e(q)} = 
\begin{cases} 
- b^v_{pq} & \mbox{if $p=e$ or $q=e$} \\
b^v_{pq} + \sgn(b^v_{pe})[b^v_{pe}b^v_{eq}]_{+}& \mbox{otherwise} 
\end{cases}
\label{BMatrixRecurrence}
\end{equation}
for $p\in I(v)\sqcup J$ and $q\in I(v)$. 
The notation $[a]_+$ means $\max(a,0)$.\\

\noindent
\textbf{Cluster mutation.} 
The clusters $X^v$ and $X^{v'}$ are related by 
\begin{equation}  
x^{v'}_{\mu_e(q)} = 
\begin{cases}
\frac{1}{x^v_q} \left(\prod_p (x^v_p)^{[b^v_{pe}]_+} + \prod_p (x^v_p)^{[-b^v_{pe}]_+} \right) &\mbox{if }q=e\\
x^v_q & \mbox{if }q\neq e
\label{ClusterRecurrenceGeom}
\end{cases}
\end{equation}
for $q\in I(v)$. 
The index $p$ runs over the set $I(v)\sqcup J$.
\\

Let $\A(\tB,X)$ be the subalgebra of $\FF$ generated by all of the cluster variables $x_e^v$, where $v$ runs over all vertices of $T$ and $e$ runs over $I(v)$, and by the $x_i$ for $i\in J$.  
This is called the \newword{cluster algebra (of geometric type)} associated to the initial seed $(\tB,X)$.
To see how the cluster algebras of geometric type are a special case of a more general construction, see \cite[Section~2]{ca4}.
For a treatment more notationally compatible with the treatment here, see \cite[Section~3.1]{framework}.

Two seeds $(\tB^v,X^v)$ and $(\tB^{v'},X^{v'})$ are \newword{equivalent} if there exists  
a bijection ${\lambda:I(v)\to I(v')}$ such that $b^{v'}_{\lambda(e)\lambda(f)}=b^v_{ef}$ for $e,f\in I(v)$, such that $b^{v'}_{j,\lambda(f)}=b^v_{jf}$
for $j\in J$ and $f\in I(v)$, and such that $x^{v'}_{\lambda(e)}=x^v_e$ for $e\in I(v)$.
Such a bijection induces, by seed mutation, a bijection from the neighbors of $v$ to the neighbors of $v'$ with each neighbor of~$v$ defining a seed equivalent to the seed at the corresponding neighbor of $v'$.
The \newword{exchange graph} $\Ex(\tB)$ is the quotient of $T$ obtained by identifying vertices $v$ and $v'$ if they define equivalent seeds, and identifying edges of $v$ with edges of $v'$ by the bijection $\lambda$.
In all of the notation defined above (e.g.\ the maps $\mu_e$), we can correctly use $\Ex(\tB)$ in place of $T$.

Specializing, suppose that $J$ is a disjoint copy of $I$ and construct an extended exchange matrix $\tB$ whose rows indexed by $I$ are $B$ and whose rows indexed by $J$ constitute an identity matrix.
A cluster algebra associated to this initial extended exchange matrix is said to have \newword{principal coefficients} at the initial seed.
Up to isomorphism, it depends only on $B$, and thus is denoted $\A_\bullet(B)$.
The associated exchange graph is $\Ex_\bullet(B)$.
In the principal coefficients case, we write each extended exchange matrix $\tB^v$ as $\twomatrix{B^v}{H^v}$ such that $B^v$ is the exchange matrix associated to $v$ and $H^v$ is a matrix with rows indexed by $I$ and columns indexed by $I(v)$.
The \newword{$\c$-vectors} at the seed $v$ are the vectors whose simple-root coordinates are given by the columns of $H^v$.
Specifically, $\c_e^v$ is the vector given by the column of $H^v$ indexed by $e\in I(v)$.

To each cluster variable $x^v_e$ in a cluster algebra with principal coefficients, there is an associated \newword{$\g$-vector} $\g^v_e$.
Here, we construct the $\g$-vectors as vectors in the weight lattice:
A cluster variable $x^{v_b}_i=x_i$ in the initial cluster has as its $\g$-vector the fundamental weight $\rho_i$.
The other $\g$-vectors are defined recursively as follows.\\

\noindent
\textbf{$\g$-vector mutation.}
Let $e$ be an edge $v$ to $v'$. 
The $\g$-vectors of the clusters $X^v$ and $X^{v'}$ are related by \begin{equation}
\g^{v'}_{\mu_e(q)} = 
\begin{cases}
-\g^v_q + \sum_{p\in I(v)} [-b^v_{pq}]_+\,\g^v_p - \sum_{i\in I}  [-b^v_{iq}]_+\,\b_i &\mbox{if }q=e\\
\g^v_q & \mbox{if }q\neq e.
\label{gRecurrence}
\end{cases}
\end{equation}
The notation $\b_i$ stands for the vector in $V^*$ with fundamental-weight coordinates given by the $i\th$ column of $B$.\\

It is not immediately obvious that this recursion yields a well-defined $\g$-vector, but $\g^v_e$ is indeed well-defined because it is the degree of $x^v_e$ with respect to a certain $\integers^n$-grading on the cluster algebra with principal coefficients.
See \cite[Sections~6--7]{ca4}.

\subsection{Frameworks and cluster algebras}\label{frame and clus sec}
We now introduce some additional conditions on frameworks and quote some key results on the connection between frameworks and combinatorial models for cluster algebras.
For a more detailed development of these and related results, see~\cite{framework}.
The results \emph{quoted} in this section (Theorems~\ref{framework exchange} and~\ref{all conj} and Corollary~\ref{cor:BdyFacet}) are stated for reflection frameworks, but are proved in \cite{framework} for the more general frameworks defined there.
In addition, two results are \emph{proved} in this section for reflection frameworks:  Proposition~\ref{rk2cyc poly}, which holds for general frameworks by essentially the same proof, and Proposition~\ref{rk2cyc reflection}, which is special to reflection frameworks.  (See Remark~\ref{rank two defn}.)

A reflection framework $(G,C)$ is \newword{complete} if $G$ has no half-edges. 
It is \newword{injective} if $v\mapsto C(v)$ is an injective map from vertices of $G$ to subsets of $\Phi$.

Define $\Cone(v)$ to be the cone $\bigcap_{e\in I(v)}\set{x\in V^*:\br{x,C\ck(v,e)}\ge 0}$ in $V^*$.
We will see soon that this notation is compatible with the earlier definition of $\Cone(v)$ as the cone spanned by the $\g$-vectors of cluster variables associated to $v\in\Ex_\bullet(B)$.
A reflection framework $(G,C)$ is \newword{polyhedral} if (1) the cones $\Cone(v)$ are the maximal cones of a fan in $V^*$ and (2) distinct vertices $v$ of $G$ define distinct cones $\Cone(v)$.
In particular, every polyhedral reflection framework is injective.
A \newword{well-connected} polyhedral reflection framework has the following additional property:
Suppose $\Cone(v)$ and $\Cone(v')$ intersect in a face $F$.
Then there is a path ${v=v_0,v_1,\ldots,v_k=v'}$ in $G$ such that $F$ is a face of $\Cone(v_i)$  for all $i$ from $0$ to $k$.

We now prepare to define a notion of simple connectivity for frameworks, which is described more fully (for general frameworks) in \cite[Section~4.2]{framework}. 
Let $(G,C)$ be a reflection framework.
Let $v_0$ be a vertex of $G$ and suppose $e_0$ and $e_1$ are two full edges incident to $v_0$. 
Let $v_{-1}$ be the vertex connected to $v_0$ by $e_0$, let $e_{-1}=\mu_{e_0}(e_1)$, let $v_{-2}$ be the vertex connected to $v_{-1}$ by $e_{-1}$, 
let $e_{-2}=\mu_{e_{-1}}(e_0)$, etc.
Similarly, let $v_1$ be the vertex connected to $v_0$ by $e_1$, let $e_2=\mu_{e_1}(e_0)$, let $e_2$ connect $v_1$ to a vertex $v_2$, etc.
We can continue in this manner as long as the edges involved are full edges.
If we never encounter half-edges (in particular, if $G$ is complete), then we have defined a path $\cdots \dashname{e_{-1}} v_{-1} \dashname{e_{0}} v_{0} \dashname{e_{1}} v_1 \dashname{e_{2}} \cdots$ that is either bi-infinite or closes up into a cycle.
We refer to $\cdots \dashname{e_{-1}} v_{-1} \dashname{e_{0}} v_{0} \dashname{e_{1}} v_1 \dashname{e_{2}} \cdots$  as a \newword{rank-two path} or \newword{rank-two cycle} accordingly.

A reflection framework $(G, C)$ is \newword{simply connected} if the fundamental group $\pi_1(G, v_0)$ is generated by paths of the form $\sigma \tau \sigma^{-1}$ where $\tau$ travels around a rank-two cycle and $\sigma$ is some path from the basepoint $v_0$ to that rank-two cycle. 
In other words, take the graph $G$ and build a regular $CW$-complex $\Sigma$ by filling in two-dimensional cells whose boundaries are the rank-two cycles. Then $(G,C)$ is simply connected if the topological space $\Sigma$ is simply connected. See the end of \cite[Section 4.2]{framework} for discussion of the relationship between $\Sigma$ and the geometry of the polyhedral fan $\F$ that occurs when $(G,C)$ is polyhedral.

Another condition called \newword{ampleness} is considered in \cite{framework}, and a framework is called \newword{exact} if it is injective and ample.
Here, we wish to avoid defining ampleness, and we can safely do so because of \cite[Proposition~4.18]{framework}, which states that if a framework is simply connected then it is ample.
Instead, we describe the role that ampleness will play below in the assertion in Theorem~\ref{framework exchange} that a complete, exact reflection framework is isomorphic to the exchange graph:
Ampleness insists that the framework not be properly covered by the exchange graph.
(Similarly, injectivity forces the exchange graph not to be properly covered by the framework.)

For the present purposes, the following two theorems (stated in the special case of reflection frameworks) are the most important results about frameworks and cluster algebras.
For statements that are closer to minimizing the hypotheses, see \cite[Sections~3--4]{framework}.
The first of the two theorems is obtained by combining \mbox{\cite[Theorem~3.25]{framework}}, \cite[Theorem~3.26]{framework}, and \cite[Theorem~4.2]{framework}.

\begin{theorem}\label{framework exchange} 
Suppose $(G,C)$ is a complete, exact reflection framework for $B$.
Then there exists a graph isomorphism $v\mapsto\Seed(v)=(\tB^v,X^v)$ from $G$ to the  principal coefficients exchange graph $\Ex_\bullet(B)$, such that the base vertex $v_b$ maps to the initial seed.
Furthermore, if $v$ is any vertex of $G$,
\begin{enumerate}
\item\label{exch mat}
The exchange matrix $B^v=[b^v_{ef}]_{e,f\in I(v)}$ associated to $\Seed(v)$ has $b^v_{ef}=\omega(C\ck(v,e),C(v,f))$.
\item \label{extended mat}
For each $e\in I(v$), the $\c$-vector $\c_e^v$ is $C(v,e)$.
\item \label{coef sign-coherent}
Each $\c$-vector $\c_e^v$ has a definite sign: it is either in the nonnegative span of the simple roots or in the nonpositive span of the simple roots.
\item \label{g vec}
If $X^v=(x^v_e:e\in I(v))$ is the cluster in $\Seed(v)$, then for each $e\in I(v$), the $\g$-vector $\g^v_e=\g(x^v_e)$ is $R(v,e)$.
\end{enumerate}
\end{theorem}

The second of the two theorems is obtained by combining \cite[Theorem~4.1]{framework}, \mbox{\cite[Corollary 4.4]{framework}}, and \cite[Corollary 4.6]{framework}.
It states that the existence of a framework with good properties implies many of the standard structural conjectures.

\begin{theorem}\label{all conj}
If a complete, exact, well-connected polyhedral reflection framework exists for $B$, then 
Conjectures~3.9--3.13, 3.15--3.18, and~3.20 of \cite{framework} all hold for principal coefficients at~$B$.  
If, in addition, a complete reflection
framework exists for $-B$, then Conjecture~3.19 of \cite{framework} also holds for principal coefficients at~$B$. 
Furthermore, the fan defined by the framework is identical to the fan defined by $\g$-vectors of clusters in $\A_\bullet(B)$. 
\end{theorem}

The results quoted here show how Corollaries~\ref{Affine exch} and~\ref{Affine g} follow from Theorem~\ref{AffineDoubleFramework}.
Specifically, Theorems~\ref{AffineDoubleFramework} and~\ref{framework exchange} immediately imply Corollary~\ref{Affine exch}.
We will construct the doubled Cambrian framework $(\DCamb_c, \DC_c)$ and the doubled Cambrian fan $\DF_c$ in such a way that $\DF_c$ is the fan defined by the framework $(\DCamb_c, \DC_c)$.
Thus Theorem~\ref{AffineDoubleFramework} and the last assertion of Theorem~\ref{all conj} immediately imply Corollary~\ref{Affine g}.  
We also have the following corollary.

\begin{cor}\label{affine type conj}
If $B$ is acyclic and $\Cart(B)$ is of affine type, then 3.9--3.20 of \cite{framework} all hold for principal coefficients at~$B$.  
\end{cor}

Almost all of the assertions of Corollary~\ref{affine type conj} follow from Theorem~\ref{all conj} and from the existence of a framework satisfying the properties established in Theorem~\ref{AffineDoubleFramework}.
The exception is the assertion that Conjecture~3.14 of \cite{framework} holds.
This is very easily argued using Theorem~\ref{framework exchange}, but requires information about the actual construction of the doubled Cambrian framework.  
We provide this argument in Section~\ref{debt paid sec}. 

We close this section with some results related to the polyhedral property and to rank-two cycles and paths.
First, a weakening of the polyhedral property holds in all frameworks, as stated in the following corollary. This result is a corollary of Proposition~\ref{dual adjacent} and appears as \cite[Corollary~2.9]{framework}. 
\begin{cor} \label{cor:BdyFacet}
Let $(G,C)$ be a reflection framework for $B$ and let $v$ and $v'$ be adjacent vertices of $G$.
Then $\Cone(v)$ and $\Cone(v')$ intersect in a common facet.
\end{cor}

Next, we prove two propositions about rank-two cycles in complete, well-con\-nected  
polyhedral reflection frameworks.
The following is a slight strengthening of \cite[Proposition 4.17]{framework}, 
stated in the special case of reflection frameworks. 

\begin{prop}\label{rk2cyc poly}  
Suppose $(G,C)$ is a complete, well-connected polyhedral reflection framework with corresponding fan $\F$.
If $\tau$ is a rank-two cycle or path, then there exists a unique codimension-$2$ face $F$ of $\F$  such that $\tau$ is the set of vertices $v$ of $G$ with $F\subset\Cone(v)$.  
If $F$ is a codimension-$2$ face of $\F$, then the set of vertices $v$ of $G$ with $F\subset\Cone(v)$ forms a rank-two cycle or path. 
\end{prop}

\begin{proof}
Suppose $\tau=(\ldots,v_{-1},v_0,v_1,\ldots)$  
is a rank-two cycle/path.
The Reflection condition in restated form implies that  
\[ \mathrm{Span}_{\RR}(C(v_i, e_i), C(v_i, e_{i+1})) = \mathrm{Span}_{\RR}(C(v_{i+1}, e_{i+1}), C(v_{i+1}, e_{i+2})) . \]
So the $2$-plane $\mathrm{Span}_{\RR}(C(v_i, e_i), C(v_i, e_{i+1}))$ is the same for all $i$ and, dually, the $(n-2)$-dimensional subspace $C(v_i, e_i)^{\perp} \cap  C(v_i, e_{i+1})^{\perp}$ is the same for all $i$. 
 Call this $(n-2)$-dimensional subspace $R$. 
 Then $\Cone(v_i) \cap R$ is an $(n-2)$-dimensional
 face of $\Cone(v_i)$. 
 Moreover, this face of $\Cone(v_i)$ is a boundary face of the common facet 
 of $\Cone(v_i)$ and $\Cone(v_{i+1})$, so that $\Cone(v_i) \cap R$ is also a face of $\Cone(v_{i+1})$, and $\Cone(v_i) \cap R = \Cone(v_{i+1}) \cap R$. 
Setting $F = \Cone(v_0) \cap R$, we see that $F$ is a face of every $\Cone(v_i)$.  

We now want to show that, if $u$ is any vertex of $G$, with $\Cone(u) \supseteq F$, then $u$ is in $\tau$. 
By well-connectedness, there is a path $u=u_0,\ldots,u_N = v_0$ such that $F$ is contained in all of the $\Cone(u_i)$. 
If $u$ is not of the form $v_i$, then there must be an index $j$ for which $u_j \not \in \{ v_i \}$ but $u_{j+1} = v_i$ for some~$i$. 
But the only neighbors of $v_i$ whose cones contain $F$ are $v_{i \pm 1}$.
This contradiction shows that $\tau$ 
is the set of vertices $v$ of $G$ with $F\subset\Cone(v)$.  

Finally, let $F$ be any codimension-$2$ face of $\F$. 
We want to show that the set ${\set{v \in G:F \subset \Cone(v)}}$ forms a rank-two path or cycle.
Let $\Cone(v_0)$ be a maximal cone of $\F$ containing $F$. 
Since $(G,C)$  
is complete, by Proposition~\ref{dual adjacent}, there are two distinct vertices $v'$ with $\Cone(v')$ adjacent to $\Cone(v_0)$ and $F\subseteq\Cone(v')$.
Write $v_{-1}$ and $v_1$ for these vertices and write $\tau=(\ldots,v_{-1},v_0,v_1,\ldots)$ 
for the rank-two cycle/path through $v_{-1}$, $v_0$, and $v_1$.   
(Since $(G,C)$  
 is complete, there are no half edges and we can continue this cycle/path indefinitely.)
Then, by the preceding paragraphs, $\tau$ is precisely the set of vertices $v$ of $G$ for which $F \subset \Cone(v)$ . \end{proof}

\begin{prop} \label{rk2cyc reflection}
Suppose $(G,C)$ is a complete, well-connected polyhedral reflection framework with corresponding fan $\F$.
Suppose $F$ is a codimension-$2$ face of $\F$ and let $\tau$ be the corresponding rank-two path or cycle.
Let $\Phi'$ be the rank-two root system $\{ \beta \in \Phi : \langle \beta, F \rangle = 0 \}$.
Then $\tau$ is a cycle if and only if $\Phi'$ is of finite type.  
\end{prop}

\begin{remark}\label{what a root system} 
The set $\Phi'$ is the set of roots of $\Phi$ contained in a $2$-dimensional linear subspace of~$V$.
The results of \cite{Deodhar,DyerReflection} show in particular that the intersection of $\Phi$ with any linear subspace is a root system in a broader sense than that of this paper (Section~\ref{frame sec}).
However, when the subspace is $2$-dimensional, the intersection is a root system in our sense.
(See \cite[Remark~2.13]{typefree} for an example of what can go wrong when the subspace has dimension greater than $2$.)
\end{remark}

\begin{remark} \label{rank two defn}  
Proposition~\ref{rk2cyc reflection} is special to reflection frameworks, and need not hold for a general framework.  
(See \cite[Remark~4.11]{framework}.) 
We emphasize that, in any framework, a rank-two path or cycle 
$\ldots,v_{-1},v_0,v_1,\ldots$ is a rank-two cycle, by definition, if it forms a finite cycle in the graph $G$, whether or not $\Phi'$ is of finite type.
\end{remark}

\begin{remark}\label{link rem}
In Proposition~\ref{rk2cyc reflection} when $\tau$ is a cycle and $\Phi'$ is of finite type, one can describe the link of $F$ in $\F$ using \cite[Lemmas~4.9--4.10]{framework}.
The link of $F$ in $\F$ is isomorphic to the $\g$-vector fan associated to an exchange matrix $B'$ such that $\Cart(B')$ is the Cartan matrix for $\Phi'$. 
\end{remark}

To prove Proposition~\ref{rk2cyc reflection}, we quote the following technical result, which is part of \cite[Corollary 4.13]{framework} restricted to reflection frameworks.

\begin{lemma} \label{Periodic}
Let $(G, C)$ be a reflection framework and let $\tau=(\ldots,v_{-1},v_0,v_1,\ldots)$ be a rank-two cycle.
Write $e$ and $f$ for the edges in $\tau$ incident to $v_0$.
Define $\beta=C(v_0,e)$ and $\gamma=C(v_0,f)$ and $b=\omega(\beta\ck,\gamma)$ and $b'=\omega(\gamma\ck,\beta)$.
Then  $\begin{bsmallmatrix} 2 & -|b'| \\ -|b| & 2 \end{bsmallmatrix}$ is a Cartan matrix of finite type. 
\end{lemma}

\begin{proof}[Proof of Proposition~\ref{rk2cyc reflection}]
If $\tau=(\ldots,v_{-1},v_0,v_1,\ldots)$ is a rank-two cycle, define $e$,~$f$, $\beta$, $\gamma$, $b$, and $b'$ as in Lemma~\ref{Periodic}.
That lemma says that the Cartan matrix ${\tilde{A} := \begin{bsmallmatrix} 2 & - |b| \\ - |b'| & 2 \end{bsmallmatrix}}$ is of finite type.  The Euler condition (E3) implies that ${K(\beta\ck,\gamma) = \pm b}$ and $K(\gamma\ck,\beta) = \pm b'$, with the same sign. 
Proposition~\ref{basis} implies that $\beta$ and $\gamma$ are an integer basis for the lattice spanned by $\Phi'$. 
Therefore $\tilde{A}$ is the Cartan matrix of $\Phi'$. 
We see that if $\tau$ is a cycle then $\Phi'$ is of finite type.

Now suppose that $\tau=(\ldots,v_{-1},v_0,v_1,\ldots)$ is an infinite path. Let $X$ be the set of roots $\beta$ such that $\beta$ lies in some $C(v_i)$ and $F \subset \beta^{\perp}$.
Since $(G,C)$ is polyhedral, the set $X$ is infinite.
However, the Reflection condition implies that $X \subset\Phi'$, and thus $\Phi'$ is infinite.
\end{proof}

\section{Doubled Cambrian frameworks}\label{double sec}
In this section, we review the definition of Cambrian frameworks, and use them to construct the doubled Cambrian framework $(\DCamb_c, \DC_c)$ for any acyclic exchange matrix $B$.
We then prove the part of Theorem~\ref{AffineDoubleFramework} that doesn't depend on the hypothesis that $\Cart(B)$ is of affine type.

\subsection{Cambrian frameworks}\label{camb sec}  
Let $B$ be an acyclic exchange matrix with Cartan companion ${A=\Cart(B)}$ and let $W$ be the associated Coxeter group.
We now review, from~\cite{framework}, the construction of a reflection framework for $B$ called the \newword{Cambrian framework}.
Cambrian frameworks are defined using a Coxeter element which encodes the orientation of $B$, as we will now explain.

A \newword{Coxeter element} of $W$ is an element $c$ that can be expressed as a product of the simple reflections $S$, in some 
 order, with each simple reflection occurring exactly once.
The information contained in $B$ is equivalent to the information $(A,c)$, where 
 $c$ is the Coxeter element obtained by multiplying $S$ in an order such that $s_i$ precedes $s_j$ whenever $b_{ij}>0$.
The acyclicity of $B$ implies that $S$ can be ordered by that rule.
Two different orders satisfying that rule multiply to the same Coxeter element because $b_{ij}=0$ implies $A_{ij}=0$, which implies that $s_i$ and $s_j$ commute.

We write $\omega_c$ and $E_c$ for the bilinear forms $\omega$ and $E$ to emphasize their dependence on $c$. 
Table~\ref{conv} emphasizes our convention on how $B$ defines a Coxeter element together with our other related conventions.
\renewcommand{\arraystretch}{1.25}
\begin{table}[ht]
\begin{center}
\begin{tabular}{| l l |}
\hline
If $b_{ij}> 0$, then &
$c=\cdots s_i \cdots s_j \cdots$.\\ 
& $E_c(\alpha_i\ck,\alpha_j)=0$.\\
& $\omega_c(\alpha_i\ck,\alpha_j)>0$.\\ \hline
\end{tabular}
\begin{tabular}{| l l |}
\hline
If $b_{ij}< 0$, then &
$c=\cdots s_j \cdots s_i \cdots$.\\ 
& $E_c(\alpha_i\ck,\alpha_j)<0$.\\
& $\omega_c(\alpha_i\ck,\alpha_j)<0$.\\ \hline
\end{tabular}
\end{center}
\smallskip
\caption{Sign conventions for Coxeter elements and bilinear forms} \label{conv}
\end{table}

An \newword{initial} letter in a Coxeter element $c$ is an element $s$ of $S$ such that $c$ has an expression as a product of the elements of $S$, beginning with $s$.
Similarly, a \newword{final} letter in $c$ is a simple reflection $s$ that can occur last in an expression for $c$ as a product of the elements of $S$.
If $s$ is an initial or final letter of $c$, then $scs$ is another Coxeter element.

A \newword{reduced word} for an element $w\in W$ is a word in the generators $S$ that is of minimal length among words for $w$.
The \newword{length} of $w$ is defined to be the length of a reduced word for
 $w$, and is written $\ell(w)$. 
The \newword{(right) weak order} is a partial order on $W$ defined as the transitive closure of all relations of the following form:  $w<ws$ for $w\in W$ and $s\in S$ with $\ell(w)<\ell(ws)$.  
The weak order is a meet semilattice in general and a lattice if $W$ is finite.
Furthermore, every nonempty set $U$ has a meet $\Meet U$, and if $U$ has an upper bound, then it has a join $\Join U$.

For each subset $J\subseteq S$, the subgroup $W_J$ of $W$ generated by $J$ is called a \newword{standard parabolic subgroup}.
The subgroup $W_J$ is a Coxeter group in its own right, with simple reflections $J$, and is an order ideal in the weak order on $W$.
The root system $\Phi_J$ associated to $W_J$ is the intersection $\Phi\cap V_J$, where $V_J$ is the span of $\set{\alpha_s:s\in J}$.
Given $w\in W$ and $J\subseteq S$, there exists a unique maximal element $w_J$ among elements of $W_J$ below $w$ in the weak order.
We call $w_J$ the \newword{projection} of $w$ to $W_J$.

A Coxeter element $c$ of $W$ induces a Coxeter element of $W_J$ called the \newword{restriction} of $c$ to $W_J$. (This is usually not the projection $c_J$.)
The restriction of $c$ to $W_J$ is obtained by deleting the letters in $S\setminus J$ from a reduced word for $c$.  
We define the notation $\br{s}$ to stand for $S\setminus\set{s}$.

We give here a recursive definition of \newword{$c$-sortable elements}, by induction on the rank of $W$
(the cardinality $n$ of $S$) and on the length of elements of $W$.  
First, the identity element is $c$-sortable for any $c$.
If $s$ is initial in $c$, then we can decide recursively whether $w$ is $c$-sortable, by considering two cases:
If $w\not\ge s$, then~$w$ is $c$-sortable if and only if it is in~$W_{\br{s}}$ and is $sc$-sortable.
If $w\ge s$, then~$w$ is $c$-sortable if and only if $sw$ is $scs$-sortable.
The condition $w\ge s$ is equivalent to $\ell(sw)=\ell(w)-1$.
The recursion involves either deciding sortability in a Coxeter group $W_{\br{s}}$ of rank $n-1$ or deciding sortability for an element $sw$ of length $\ell(w)-1$, and thus terminates.
The consistency of these recursive requirements follows from a non-recursive definition of $c$-sortable elements, found for example in~\cite[Section~5.1]{framework}. 

The vertices of the $c$-Cambrian framework are the $c$-sortable elements.
Before defining the quasi-graph structure on $c$-sortable elements, we recursively define a label set for each $c$-sortable element.
Let $v$ be $c$-sortable and let $s$ be initial in $c$.
Define  
\begin{equation}\label{C recur}
C_c(v)=\left\lbrace\begin{array}{ll}
C_{sc}(v)\cup\set{\alpha_s}&\mbox{if } v\not\ge s\\
s C_{scs}(sv)&\mbox{if } v\ge s
\end{array}\right.
\end{equation}
The set $C_{sc}(v)$ is a set of roots in $\Phi_\br{s}$, defined by induction on the rank of $W$.
The set $C_{scs}(sv)$ is defined by induction on the length of $v$.  
The base of the inductive definition is that the unique element of the trivial Coxeter group has empty label set or, for those who dislike such stark minimalism, that the identity is labeled by the set of simple roots.

The non-recursive definition of $c$-sortable elements constructs a special reduced word for each $c$-sortable element, called its \newword{$c$-sorting word}.  
The set $C_c(v)$ can also be defined non-recursively in terms of the combinatorics of the $c$-sorting word for $v$.
For details, see \cite[Section~5]{typefree} or \cite[Section~5.1]{framework}.

The quasi-graph structure on $c$-sortable elements is obtained from the restriction of the weak order to $c$-sortable elements.
This restriction is called the \newword{$c$-Cambrian semilattice} and denoted $\Camb_c$.
The $c$-Cambrian semilattice is a sub-meet-semilattice \cite[Theorem~7.1]{typefree} of the weak order on $W$.
We will also use the symbol $\Camb_c$ for the undirected Hasse diagram of the $c$-Cambrian semilattice.

The following lemma, which is part of \cite[Lemma~5.11]{framework}, is the key to constructing and labeling an $n$-regular quasi-graph from the $c$-Cambrian semilattice.  
The symbol $\covered$ denotes a cover relation.
\begin{lemma}\label{cov beta part 1}
If $v'\covered v$ in the $c$-Cambrian semilattice, then there exists a unique root $\beta$ such that $\beta\in C_c(v')$ and $-\beta\in C_c(v)$.
The root $\beta$ is positive.
\end{lemma}

Suppose $v$ is a $c$-sortable element.
In light of Lemma~\ref{cov beta part 1}, we can label each incident pair $(v,e)$, where $e$ is an edge $v'\covered v$ in $\Camb_c$, 
by the unique negative root $\beta\in C_c(v)$ such that $-\beta\in C_c(v')$.
Similarly, we can label each incident pair $(v,e)$, where $e$ is an edge $v\covered v'$ in $\Camb_c$ by the unique positive root $\beta\in C_c(v)$ such that $-\beta\in C_c(v')$.
As explained in the paragraphs before \cite[Theorem~5.12]{framework}, the vertices in the graph $\Camb_c$ all have degree at most $n$.  
To each vertex, we add the appropriate number of half-edges, to make an $n$-regular quasi-graph.
The new half-edges on $v$ get the roots in $C_c(v)$ that were not assigned to full edges. 
We will need one more condition on a framework.\\

\noindent
\textbf{Full edge condition:}  \label{Full edge condition}
If $(v,e)$ is an incident pair and the sign of $C(v,e)$ is negative, then $e$ is a full edge.\\

Most of the following theorem is stated as \cite[Corollary~5.14]{framework}, but the assertion about the Full edge condition follows from \cite[Theorem~5.12]{framework}.
(That theorem states that $(\Camb_c,C_c)$ is a ``descending'' framework, and the Full edge condition is part of the definition of a descending framework.)

\begin{theorem}\label{camb good}
The pair $(\Camb_c,C_c)$ is an exact, well-connected, polyhedral, simply connected reflection framework for $B$ and satisfies the Full edge condition.
\end{theorem}

For each $c$-sortable element $v$, define a cone  
\[\Cone_c(v)=\bigcap_{\beta \in C_c(v)}\set{x\in V^*:\br{x,\beta} \geq 0}.\]
The cone $\Cone_c(v)$ coincides with the cone $\Cone(v)$ defined in Section~\ref{frame and clus sec}.
The \newword{$c$-Cambrian fan} $\F_c$ is the collection consisting of the cones $\Cone_c(v)$ and all of their faces, with $v$ running over all $c$-sortable elements. 
The following result is part of \cite[Corollary~5.15]{framework}.

\begin{theorem}\label{camb fan}
The $c$-Cambrian fan $\F_c$ is a simplicial fan, and distinct vertices of $\Camb_c$ label distinct maximal cones of $\F_c$.
\end{theorem}

We now fill in some additional results on sortable elements, Cambrian semilattices, and the Cambrian framework.
The following propositions are \cite[Proposition~2.30]{typefree} and \cite[Proposition~3.13]{typefree}.
\begin{prop}\label{sort para easy}
Let $J \subseteq S$ 
and let $c'$ be the restriction of $c$ to $W_J$. 
Then an element $v \in W_J$ is $c$-sortable if and only if it is $c'$-sortable. 
\end{prop}

\begin{prop}\label{sort para}
If $v$ is $c$-sortable and $J\subseteq S$, then $v_J$ is $c'$-sortable, where $c'$ is the restriction of $c$ to $W_J$.
\end{prop}
Here $v_J$ is the projection of $v$ to the standard parabolic subgroup $W_J$.
See Section~\ref{camb sec}.  

The following theorems are \cite[Theorem~7.1]{typefree} and \cite[Theorem~7.4]{typefree}.
\begin{theorem} \label{meets and joins}
Let~$U$ be a collection of $c$-sortable elements of~$W$. 
If~$U$ is nonempty then $\Meet U$ is $c$-sortable. 
If~$U$ has an upper bound then $\Join U$ is $c$-sortable.
\end{theorem}

\begin{theorem}\label{c to scs}
If~$s$ is initial in~$c$ then the map
\[v\mapsto\left\lbrace\begin{array}{cl}
sv&\mbox{if }v\ge s\\
s\join v&\mbox{if }v\not\ge s
\end{array}\right.\]
is a bijection from the set $\set{\mbox{$c$-sortable elements $v$ such that $s\join v$ exists}}$ to the set of all $scs$-sortable elements.
The inverse map~is
\[x\mapsto\left\lbrace\begin{array}{cl}
sx&\mbox{if }x\not\ge s\\
x_{\br{s}}&\mbox{if }x\ge s.
\end{array}\right.\]
\end{theorem}

An \newword{inversion} of $w\in W$ is a reflection $t$ such that $\ell(tw)<\ell(w)$. 
Inversions interact nicely with the map $w\mapsto w_J$ for any $J\subseteq S$.
Specifically, $\inv(w_J)=\inv(w)\cap W_J$.

A \newword{cover reflection} of $w\in W$ is an inversion $t$ of $w$ such that $tw=ws$ for some $s\in S$.
The name refers to the fact that $tw\covered w$ in the weak order if and only if $t$ is a cover reflection of $w$.
We write $\inv(w)$ for the set of inversions of $w$ and $\cov(w)$ for the set of cover reflections of~$w$.
We have $\inv(tw)=\inv(w)\setminus\set{t}$ when $t$ is a cover reflection of~$w$.

The following propositions are first, \cite[Proposition~5.2]{typefree}, second, the combination of parts of \cite[Proposition~5.1]{typefree} and \cite[Proposition~5.4]{typefree}, and third, the combination of parts of \cite[Proposition~5.1]{typefree} and \cite[Proposition~5.3]{typefree}.

\begin{prop}\label{lower walls}
Let $v$ be a $c$-sortable element.
The set of negative roots in $C_c(v)$ is $\set{-\beta_t:t\in\cov(v)}$.
\end{prop}

\begin{prop}\label{Cc initial}
Let~$s$ be initial in~$c$ and let $v$ be a $c$-sortable element of~$W$ such that~$s$ is a cover reflection of $v$.
Then $\cov(v)=\set{s}\cup\cov(v_{\br{s}})$, and the set of positive roots in $C_c(v)$ is obtained by applying the reflection $s$ to each positive root in $C_{sc}(v_{\br{s}})$.
\end{prop}

\begin{prop}\label{Cc final}
Let~$s$ be final in~$c$ and let $v$ be $c$-sortable with $v \geq s$.
Then $v=s\join v_{\br{s}}$ and $C_c(v)=C_{cs}(v_{\br{s}})\cup\set{-\alpha_s}$.
\end{prop}

We will also need a certain downward projection map from $W$ to $c$-sortable elements in $W$. 
Given $w\in W$, by \mbox{\cite[Corollary~6.2]{typefree}}, there is an element $\pidown^c(w)$ that is maximal among the $c$-sortable elements that are $\leq w$.
The map $\pidown^c$ lets us describe the cover relations in $\Camb_c$.
The description is proved as \cite[Lemma~5.9]{framework}. 
\begin{lemma}\label{Cambrian covers}
Let $v$ be a $c$-sortable element.
Then $v'\covered v$ in the $c$-Cambrian semilattice if and only if $v'=\pidown^c(tv)$ for some $t\in\cov(v)$.
Furthermore if $t_1$ and $t_2$ are distinct cover reflections of $v$, then $\pidown^c(t_1v)\neq\pidown^c(t_2v)$.
\end{lemma}

Recall that Lemma~\ref{cov beta part 1} is part of \cite[Lemma~5.11]{framework}.  
The rest of \mbox{\cite[Lemma~5.11]{framework}} gives some more detail:  
\begin{lemma}\label{cov beta part 2}
If $v'\covered v$ in the $c$-Cambrian semilattice, then the unique root $\beta$ such that $\beta\in C_c(v')$ and $-\beta\in C_c(v)$ is the positive root $\beta_t$ associated to the cover reflection $t$ of $v$ such that $v'=\pidown^c(tv)$.
\end{lemma}

The following theorem is \cite[Theorem~6.1]{typefree}.
\begin{theorem} \label{order preserving}
$\pidown^c$ is order preserving.
\end{theorem}

Let $D$ be the cone $\bigcap_{s \in S} \set{x\in V^*: \br{x,\alpha_s}\ge 0}$.
The cones $wD$, for $w\in W$, are distinct, and form a fan whose support is called the \newword{Tits cone}.  
(Here the action of $W$ on $V^*$ is dual to its action on $V$.)

The following theorem is \cite[Theorem~6.3]{typefree}.
\begin{theorem} \label{pidown fibers}
Let $v$ be $c$-sortable. Then $\pidown^c(w)=v$ if and only if $w D \subseteq \Cone_c(v)$.
\end{theorem}

We record the following simple observation.
\begin{prop}\label{pidown commutes}
If $w\in W$ and $J\subseteq S$, then $\pidown^c(w_J)=(\pidown^c(w))_J$.
\end{prop}
\begin{proof}
Since $w\ge w_J$, Theorem~\ref{order preserving} says that $\pidown^c(w)\ge \pidown^c(w_J)$.
Since $x\mapsto x_J$ is order preserving, since $\pidown^c(w_J)\in W_J$, and since $W_J$ is an order ideal in the weak order on $W$, 
we have $(\pidown^c(w))_J\ge (\pidown^c(w_J))_J=\pidown^c(w_J)$.
On the other hand, since $\pidown^c(w)\le w$, we have $(\pidown^c(w))_J\le w_J$, and $(\pidown^c(w))_J$ is $c$-sortable by Propositions~\ref{sort para easy} and~\ref{sort para}.
But $\pidown^c(w_J)$ is the maximal $c$-sortable element below $w_J$, so $\pidown^c(w_J)\ge(\pidown^c(w))_J$.
\end{proof}

Let $w\in W$.
Earlier, we used the notation $\inv(w)$ to denote the inversions of $W$ as a set of reflections.
It is convenient to use the bijection between reflections and positive roots to also think of $\inv(w)$ as a set of positive roots.
We observe that $\inv(w)$ is the set of positive roots $\beta$ such that $wD$ is contained in $\set{x\in V^*:\br{x,\beta}\le 0}$.

Let $\beta$ and $\gamma$ be distinct, positive, real roots in $\Phi$ such that all of the positive roots in the linear span of $\beta$ and $\gamma$ are in the \emph{nonnegative} 
span of $\beta$ and $\gamma$.
We call $\Phi'=\Span_{\RR}(\beta,\gamma)\cap\Phi$ a \newword{generalized rank-two parabolic} sub root system of $\Phi$ and call $\beta$ and $\gamma$ the \newword{canonical roots} of $\Phi'$.
An element $w\in W$ is \newword{$c$-aligned with respect to $\Phi'$} if one of the following cases holds:
\begin{enumerate}
\item[(i)] $\omega_c(\beta,\gamma)=0$ and $\inv(w)\cap\Phi'$ is $\emptyset$ or $\set{\beta}$ or $\set{\gamma}$ or $\set{\beta, \gamma}$.
\item[(ii)] $\omega_c(\beta,\gamma)>0$ and either $\gamma\not\in\inv(w)$ or $\inv(w)\cap\Phi'=\set{\gamma}$ or $\inv(w) \supseteq \Phi'$.
\item[(iii)] $\omega_c(\beta,\gamma)<0$ and either $\beta\not\in\inv(w)$ or $\inv(w)\cap\Phi'=\set{\beta}$  or $\inv(w) \supseteq \Phi'$.
\end{enumerate} 
The last option in Case~(i) can only occur if $\beta$ and $\gamma$ are perpendicular.
However, it is possible to have $\omega_c(\beta,\gamma)=0$ even when $\beta$ and $\gamma$ are not perpendicular.
(For example, looking forward to Examples~\ref{affineG2 ex} and~\ref{affineG2 slice ex}, define $\beta=s_3s_1\alpha_2$ and $\gamma=s_2s_2\alpha_2$.
Then $\omega_c(\beta\ck,\gamma)=\omega_c(3\alpha_i\ck+\alpha_2\ck+3\alpha_3\ck,\alpha_1+2\alpha_2)=0$, but $K(\beta\ck,\gamma)=-2$.) 
The last options in Cases~(ii) and~(iii) can only occur if $\Phi'$ is finite.

The following is a slight weakening of \cite[Theorem~4.3]{typefree}.

\begin{theorem} \label{sortable is aligned}
An element~$w$ of~$W$ is $c$-sortable if and only if $w$ is $c$-aligned with respect to every generalized rank-two parabolic sub root system of $\Phi$.
\end{theorem}

We now prove three key lemmas that relate the cones $\Cone_c(v)$ to the hyperplanes~$\alpha_s^\perp$.
Given a set $U\subseteq V^*$ and a root $\beta\in\Phi$, 
we say that $U$ is \newword{above} the hyperplane $\beta^\perp$ if every point in $U$ is either in $\beta^\perp$ or on the opposite side of $\beta^\perp$ from $D$.
If $U$ is above $\beta^\perp$ and does not intersect $\beta^\perp$, then we say that $U$ is \newword{strictly above} $\beta^\perp$.
Similarly, $U$ is \newword{below} $\beta^\perp$ if every point in $U$ is either contained in $\beta^\perp$ or on the same side of $\beta^\perp$ as $D$, and $U$ is \newword{strictly below} $\beta^\perp$ if $U$ is below $\beta^\perp$ and does not intersect $\beta^\perp$.

\begin{lemma} \label{AboveBelow}
Let $c$ be a Coxeter element, let $s\in S$, and let $v$ be $c$-sortable.
If $v \not\geq s$ then $\Cone_c(v)$ is below $\alpha_s^\perp$, and if $v \geq s$ then $\Cone_c(v)$ is above $\alpha_s^\perp$.  
\end{lemma}

We prove Lemma~\ref{AboveBelow} by quoting two nontrivial results and giving an easy argument.  
Proposition~\ref{camb g sign-coherent} is \cite[Corollary~5.21]{framework}, 
and Proposition~\ref{extreme} appears as the second assertion of \cite[Corollary~5.15]{framework}. 
See also Theorem~\ref{framework exchange}\eqref{g vec}.

\begin{prop}\label{camb g sign-coherent}
Given any vertex $v$ in the Cambrian framework $(\Camb_c,C_c)$ and any $s\in S$, all of the $\g$-vectors of cluster variables in $\Seed(v)$ are weakly on the same side of the hyperplane $\alpha_s^\perp$ in $V^*$.
\end{prop}

\begin{prop}\label{extreme}
 Given any vertex $v$ in the Cambrian framework $(\Camb_c,C_c)$, the extreme rays of the cone $\Cone_c(v)$ are spanned by the $\g$-vectors of cluster variables in $\Seed(v)$.
\end{prop}

\begin{proof}[Proof of Lemma~\ref{AboveBelow}]  
Propositions~\ref{extreme} and~\ref{camb g sign-coherent} combine to establish the following assertion:
For any $c$-sortable element $v$ and any $s\in S$, the cone $\Cone_c(v)$ is either above or below $\alpha_s^\perp$.
If $v\not\ge s$ then $vD$ is below $\alpha_s^\perp$, so Theorem~\ref{pidown fibers} implies that $\Cone_c(v)$ is below $\alpha_s^\perp$.
Similarly, if $v\ge s$, then $vD$ is above $\alpha_s^\perp$, so Theorem~\ref{pidown fibers} implies that $\Cone_c(v)$ is above $\alpha_s^\perp$.
\end{proof}

\begin{lemma} \label{PrettyLemma}
Let $s$ be initial in $c$. Let $v$ be $c$-sortable with $v \geq s$. 
Then ${\Cone_c(v) \cap \alpha_s^\perp}$ is a face of $\Cone_c(v_{\br{s}})$.
\end{lemma}

\begin{proof}
Proposition~\ref{sort para} says that $v_\br{s}$ is $sc$-sortable, and therefore it is $c$-sortable by Proposition~\ref{sort para easy}.
In particular, the notation $\Cone_c(v_{\br{s}})$ makes sense.

By Lemma~\ref{AboveBelow}, $\Cone_c(v)$ is above $\alpha_s^\perp$.
In particular, $\Cone_c(v) \cap \alpha_s^\perp$ is a face of $\Cone_c(v)$, which we call $F$.
We first claim that $F\subseteq\Cone_c(v_\br{s})$.
To prove the claim, we argue by induction on $\ell(v)$ and break into two cases.
The base case $v=s$ falls into Case 1, where we do not appeal to induction.

\noindent
\textbf{Case 1:} $F$ is a facet of $\Cone_c(v)$. 
Then $-\alpha_s\in C_c(v)$, so $s$ is a cover reflection of $v$ by Proposition~\ref{lower walls}. 
Equation \eqref{C recur} says that $C_c(v_{\br{s}})=C_{sc}(v_{\br{s}})\cup\set{\alpha_s}$.
By Proposition~\ref{Cc initial}, each element of $C_{sc}(v_{\br{s}})$ either coincides with an element of $C_c(v)\setminus\set{-\alpha_s}$ or is related to an element of $C_c(v)\setminus\set{-\alpha_s}$ by the reflection $s$.
Thus the defining inequalities of $\Cone_c(v)$ and $\Cone_c(v_\br{s})$ coincide for points on $\alpha^\perp$, so the claim is true in this case.

\noindent
\textbf{Case 2:} $\dim(F)<n-1$.
Write $F=\Cone_c(v)\cap(\beta_1^{\perp}\cap\cdots\cap\beta_k^{\perp})$ for distinct $\beta_1,\ldots,\beta_k\in C_c(v)$.
Let $y$ be in the relative interior of $F$ and let $x$ be in the interior of $D$.
If every $\beta_i$ is positive, then in particular, $x$ is strictly below each $\beta_i^\perp$.
But also, $\Cone_c(v)$ is weakly below each $\beta_i^\perp$, and therefore for small enough $\ep>0$, the point $y+\ep x$ is in the interior of $\Cone_c(v)$.
But $y+\ep x$ is strictly below $\alpha_s^\perp$, and this contradiction to Lemma~\ref{AboveBelow} shows we can choose $i$ such that $\beta_i=-\beta_t$ for some reflection $t$.
Let $u = \pidown^c(tv)$, which, by Lemma~\ref{Cambrian covers}, is covered by $v$ in $\Camb_c$.
Corollary~\ref{cor:BdyFacet} (in light of Theorem~\ref{camb good}) says that $\Cone_c(v)$ and $\Cone_c(u)$ intersect in a facet, so that in particular $F$ is a face of $\Cone_c(u)$.
By induction, $F'=\Cone_c(v') \cap \alpha_s^\perp$ is contained in $\Cone_c(u_{\br{s}})$, but $F\subseteq F'$, so $F\subseteq\Cone_c(u_{\br{s}})$.

Now, Proposition~\ref{pidown commutes} says that $u_{\br{s}}=\pidown^c((tv)_{\br{s}})$ and that $v_{\br{s}}=(\pidown^c(v))_{\br{s}}=\pidown^c(v_{\br{s}})$.
If $(tv)_{\br{s}}=v_{\br{s}}$, then $u_{\br{s}}=v_{\br{s}}$, and we are done.
If not, then 
since $\inv(v_\br{s})=\inv(v)\cap W_\br{s}$ and $\inv ((tv)_\br{s})=\inv(tv)\cap W_\br{s}$, and since $\inv(tv)=\inv(v)\setminus\set{t}$, we conclude that $\inv((tv)_\br{s})=\inv(v_\br{s})\setminus\set{t}$.
In particular, the reflection $t$ is in $W_{\br{s}}$, and $t(v_{\br{s}})=(tv)_{\br{s}}$ is covered by $v_{\br{s}}$.
Thus $u_{\br{s}}\covered v_{\br{s}}$ in $\Camb_c$ by Lemma~\ref{Cambrian covers}, and Corollary~\ref{cor:BdyFacet} says that $\Cone_c(u_{\br{s}})$ and $\Cone_c(v_{\br{s}})$ intersect in a mutual facet defined by   $\beta_t^\perp$.
Therefore $F\subseteq\Cone_c(v_{\br{s}})$.

Having proven the claim, we prove the lemma.
Since $\Cone_c(v) \cap \alpha_s^\perp\subseteq\Cone_c(v_\br{s})$, the intersection $\Cone_c(v) \cap \alpha_s^\perp$ equals $(\Cone_c(v)\cap\Cone_c(v_\br{s}))\cap(\Cone_c(v_\br{s}) \cap \alpha_s^\perp)$. 
Theorem~\ref{camb fan} implies that $\Cone_c(v)\cap\Cone_c(v_\br{s})$ is a face of $\Cone_c(v_\br{s})$.
By Lemma~\ref{AboveBelow}, $\Cone_c(v_\br{s})$ is below $\alpha_s^\perp$, so $\Cone_c(v_\br{s}) \cap \alpha_s^\perp$ is a face of $\Cone_c(v_\br{s})$.
We have written $\Cone_c(v) \cap \alpha_s^\perp$ as an intersection of faces of $\Cone_c(v_\br{s})$.
\end{proof}

\begin{lemma} \label{DualPrettyLemma}
Let $s$ be final in $c$. 
Let $v$ be $c$-sortable with $v \not \geq s$. 
Then $s \join (sv)_{\br{s}}$ exists, and is $c$-sortable, and $\Cone_c(v) \cap \alpha_s^\perp$ is a face of $\Cone_c(s \join (sv)_{\br{s}})$.
\end{lemma}

\begin{proof}
Since $sv \geq s$ and $sv > (sv)_{\br{s}}$, we know that $s \join (sv)_{\br{s}}$ exists. 
By the recursive definition of sortability, the element $sv$ is $scs$-sortable, so $(sv)_{\br{s}}$ is $cs$-sortable by Proposition~\ref{sort para}, and thus $c$-sortable by Proposition~\ref{sort para easy}.
Now $s \join (sv)_{\br{s}}$ is $c$-sortable by Theorem~\ref{meets and joins}.  

By Lemma~\ref{PrettyLemma}, $\Cone_{scs}(sv) \cap \alpha_s^\perp$ is a face of $\Cone_{scs}((sv)_{\br{s}})$. 
Because $\Cone_c(v)$ equals $s \Cone_{scs}(sv)$ by \eqref{C recur}, 
the intersection $\Cone_c(v)\cap\alpha_s^\perp$ equals $\Cone_{scs}(sv)\cap\alpha_s^\perp$.
Thus $\Cone_c(v)\cap\alpha_s^\perp$ is a face of $\Cone_{scs}((sv)_{\br{s}})$.

By \eqref{C recur}, we have $C_{scs}((sv)_{\br{s}})=C_{cs}((sv)_{\br{s}})\cup\set{\alpha_s}$. 
On the other hand, Proposition~\ref{Cc final} (with $s \join (sv)_{\br{s}}$ replacing $v$) says that $C_c(s \join (sv)_{\br{s}})$ equals ${C_{cs}((sv)_{\br{s}})\cup\set{-\alpha_s}}$.  
We see that $\Cone_{scs}((sv)_{\br{s}})$ and $\Cone_c(s\join(sv)_{\br{s}})$ intersect in a common facet $F$ defined by $\alpha_s^\perp$.
Since $\Cone_c(v)\cap\alpha_s^\perp$ is a face of $\Cone_{scs}((sv)_{\br{s}})$ contained in $\alpha_s^\perp$, it is also a face of $F$, and thus a face of $\Cone_c(s\join(sv)_{\br{s}})$.
\end{proof}

\subsection{Doubling the Cambrian fan}\label{DFc sec} 
We construct the doubled Cambrian framework by ``gluing'' together two opposite Cambrian frameworks.
We begin by doubling the Cambrian fan.

We continue the notation $B$ for an acyclic exchange matrix with $A=\Cart(B)$, 
Coxeter group $W$, and Coxeter element $c$.
The notation $\Tits(A)$ denotes the Tits cone and $-\Tits(A)$ is its negation.
We write $\DF_c$ for the 
collection consisting of the cones in the Cambrian fan $\F_{c}$ and the negations of cones in the Cambrian fan $\F_{c^{-1}}$.
In symbols, $\DF_c=\F_c\cup(-\F_{c^{-1}})$.
We call $\DF_c$ the \newword{doubled Cambrian fan}  
and justify this name in Theorem~\ref{DoubleFan}.

\begin{example}\label{affineG2 ex}
We consider the doubled Cambrian fan for the exchange matrix $B=\begin{bsmallmatrix*}[r]0&1&1\\-3&0&0\\-1&0&0 \end{bsmallmatrix*}$.
The corresponding Coxeter group is of type $\widetilde{G}_2$ and the corresponding Coxeter element is $c=s_1s_2s_3$.
(Note that the indices $1$, $2$, and $3$ do \emph{not} give the standard linear numbering of the Dynkin diagram.)
Figure~\ref{G2tildeCamb pos} shows the $c$-Cambrian fan $\F_c$ for this example.
(Compare \cite[Figure~6]{typefree}.)
\begin{figure}
\scalebox{0.87}{\includegraphics{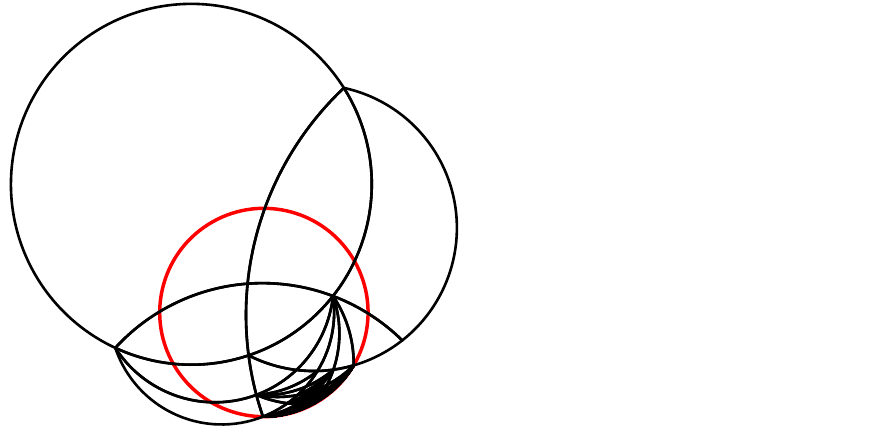}}
\caption{The $c$-Cambrian fan for $B$ as in Example~\ref{affineG2 ex}}
\label{G2tildeCamb pos}
\bigskip
\bigskip
\bigskip
\scalebox{0.87}{\includegraphics{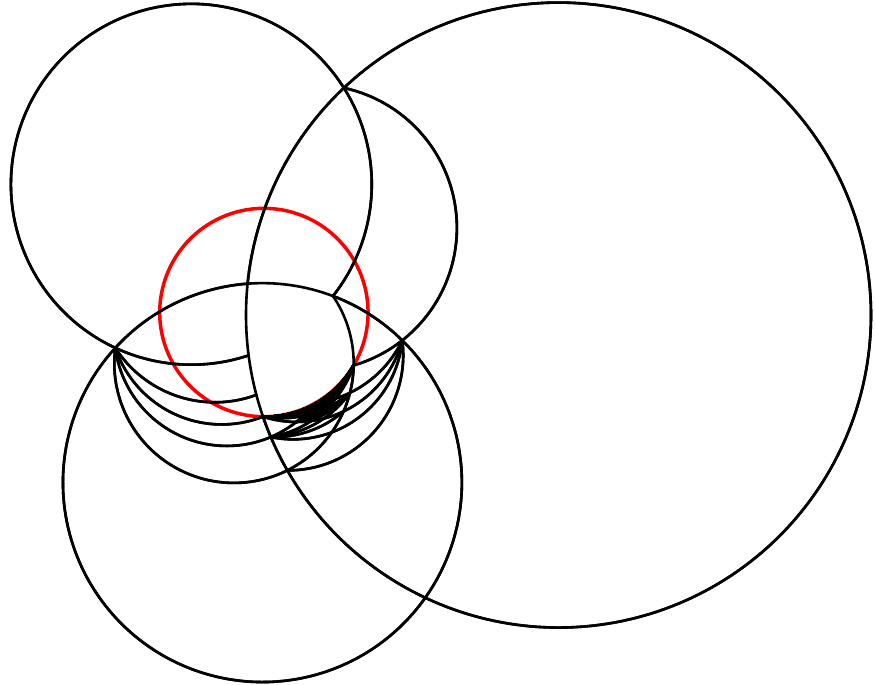}}
\caption{The antipodal $c^{-1}$-Cambrian fan for $B$ as in Example~\ref{affineG2 ex}}
\label{G2tildeCamb neg}
\bigskip
\bigskip
\bigskip
\scalebox{0.87}{\includegraphics{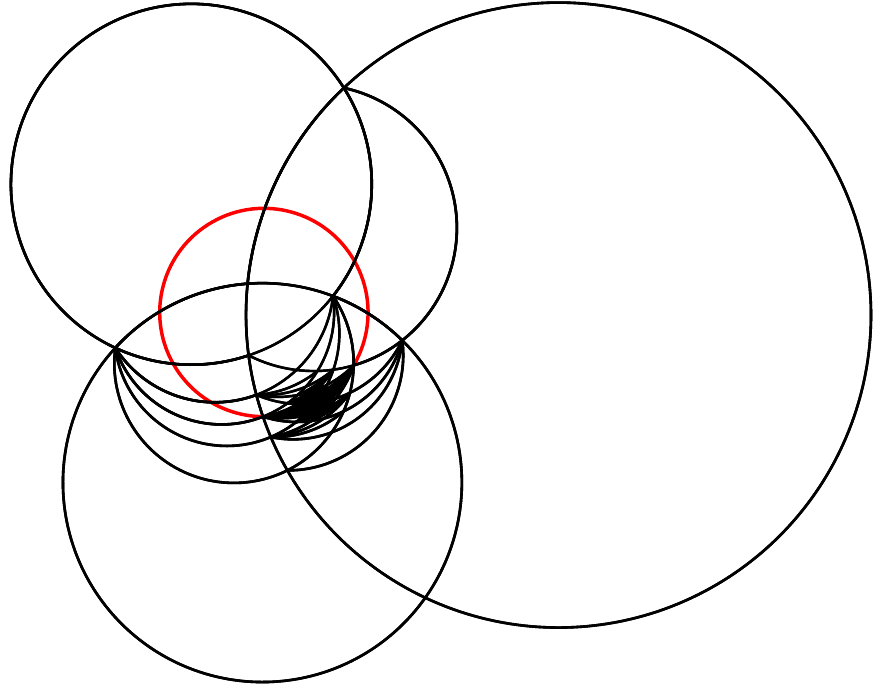}}
\caption{The doubled Cambrian fan for $B$ as in Example~\ref{affineG2 ex}}
\label{G2tildeCamb both}
\end{figure}
The picture should be interpreted as follows:  The $c$-Cambrian fan, intersected with a unit sphere about the origin, is a collection of spherical triangles and their faces.
The picture shows this collection of triangles stereographically projected to the plane.
The red circle marks the boundary of the Tits cone.
Figure~\ref{G2tildeCamb neg} is a similar projection of the antipodal image $-\F_{c^{-1}}$ of the $c^{-1}$-Cambrian fan.
Figure~\ref{G2tildeCamb both} shows the doubled Cambrian fan $\DF_c$, the union of the two fans pictured in Figures~\ref{G2tildeCamb pos} and~\ref{G2tildeCamb neg}.
In each of Figures~\ref{G2tildeCamb neg} and~\ref{G2tildeCamb both}, the triangular exterior region in the stereographic projection represents a cone in the fan.
Figure~\ref{G2tildeCamb both alt} shows the same doubled Cambrian fan in a different stereographic projection.
The boundary of the Tits cone is the red vertical line in the center and the Tits cone projects to the left of the line.
\end{example}
\begin{figure}
\scalebox{0.7}{\includegraphics{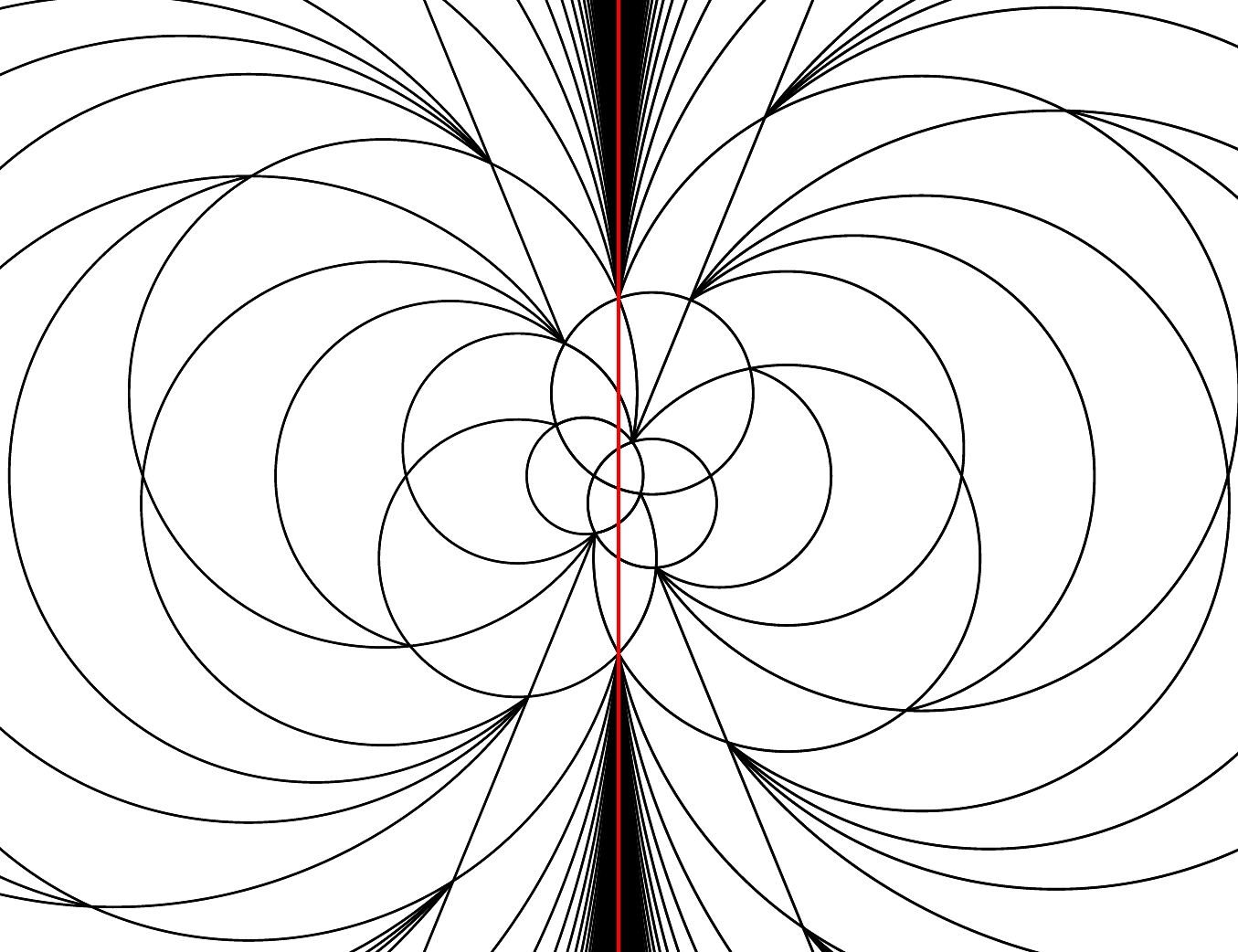}}
\caption{Another view of the doubled Cambrian fan}
\label{G2tildeCamb both alt}
\end{figure}

\begin{remark}\label{visualization preferences}
Preferences for visualizing $3$-dimensional fans vary, even among coauthors.
Readers who find the stereographic projections referenced in Example~\ref{affineG2 ex} difficult to interpret will find a visualization of the same example, via affine slices, in Example~\ref{affineG2 slice ex}.
\end{remark}

\begin{theorem} \label{DoubleFan}
For any acyclic exchange matrix $B$, the collection $\DF_c$ of cones is a simplicial fan.
\end{theorem}

\begin{proof}
Once we know that $\DF_c$ is a fan, the assertion that the fan is simplicial is immediate from the definition, in light of Theorem~\ref{camb fan}.
As explained in Section~\ref{poly sec}, to see that $\DF_c$ is a fan, we only need to check that the maximal cones meet nicely. 
Moreover, we already know (Theorem~\ref{camb fan}) that two maximal cones of $\F_c$ meet nicely, as do two cones of $-\F_{c^{-1}}$. 
So we need to show that, if $v$ is $c$-sortable and $u$ is $c^{-1}$-sortable, then $\Cone_c(v)$ and $-\Cone_{c^{-1}}(u)$ meet nicely. 

We argue by induction first on the rank of $W$ and then, holding the rank constant, on $\ell(v)$.
Our base case is the vacuous case where $W$ has rank $0$. If the rank is positive but $v=e$, so $\ell(v)=0$, this is not a base case but rather a case where we induct on rank; this scenario is covered by Cases~2 and~4 below.

We now consider the inductive case where the rank of $W$ be positive, so $c$ is not $e$, and let $s$ be initial in $c$.  
We break into four cases.

\textbf{Case 1:} $v \geq s$ and $u \not \geq s$. Then $\Cone_c(v) = s \Cone_{scs}(sv)$ and $-\Cone_{c^{-1}}(u) = - s\Cone_{s c^{-1} s}(s u)$ by \eqref{C recur}. 
By induction on $\ell(v)$ we see that $\Cone_{scs}(sv)$ and $- \Cone_{s c^{-1} s}(s u)$ meet nicely, so $\Cone_c(v)$ and $-\Cone_{c^{-1}}(u)$ meet nicely. 

\textbf{Case 2:} $v \not \geq s$ and $u \geq s$. 
Then $v$ is in $W_{\br{s}}$ and $C_c(v)=C_{sc}(v)\cup\set{\alpha_s}$ by \eqref{C recur}.
Also, $u = s \join u_{\br{s}}$ and $-C_{c^{-1}}(u)=-C_{c^{-1}s}(u_{\br{s}})\cup\set{\alpha_s}$ by Proposition~\ref{Cc final}.
By induction on rank, $\Cone_{sc}(v)$ and $- \Cone_{c^{-1} s}(u_{\br{s}})$ meet nicely in $V_{\br{s}}$. 
Therefore the cones in $V$ defined by $C_{sc}(v)$ and $-C_{c^{-1}s}(u_{\br{s}})$ also meet nicely, and this nice meeting is preserved by intersecting with the halfspace consisting of points below~$\alpha_s^\perp$.

\textbf{Case 3:} $v \geq s$ and $u \geq s$. 
Then $\Cone_c(v)$ is above $\alpha_s^\perp$ and $-\Cone_{c^{-1}}(u)$ is below $\alpha_s^\perp$, by Lemma~\ref{AboveBelow}, used twice.
Thus
\begin{equation}\label{case3eq}
\Cone_c(v) \cap ( -\Cone_{c^{-1}}(u) ) = \left( \Cone_c(v) \cap \alpha_s^\perp \right) \cap ( -\Cone_{c^{-1}}(u) ).
\end{equation}
We now apply the Case 2 argument to $(v_{\br{s}}, u)$ to conclude that $\Cone_c(v_{\br{s}})$ and $- \Cone_{c^{-1}}(u)$ meet nicely.
This uses induction on rank.
By Lemma~\ref{PrettyLemma}, $\Cone_c(v) \cap \alpha_s^\perp$ is a face of $\Cone_c(v_{\br{s}})$, so $\Cone_c(v) \cap \alpha_s^\perp$ and $- \Cone_{c^{-1}}(u)$ also meet nicely.  
That is, $\left( \Cone_c(v) \cap \alpha_s^\perp \right)$ is a face of  $\Cone_c(v) \cap \alpha_s^\perp$ and of $- \Cone_{c^{-1}}(u)$. 
But then
 by \eqref{case3eq}, $\Cone_c(v) \cap ( -\Cone_{c^{-1}}(u) ) $ is a face of  $\Cone_c(v) \cap \alpha_s^\perp$ and of $- \Cone_{c^{-1}}(u)$. 
Since $\Cone_c(v) \cap \alpha_s^\perp$ is a face of $\Cone_c(v)$, we see that  $\Cone_c(v) \cap ( -\Cone_{c^{-1}}(u) ) $ is a face of $\Cone_c(v)$ as well.

\textbf{Case 4:} $v \not \geq s$ and $u \not \geq s$. 
Then $\Cone_c(v)$ is below $\alpha_s^\perp$ and $-\Cone_{c^{-1}}(u)$ is above, so 
\begin{equation}\label{case4eq}
\Cone_c(v) \cap ( -\Cone_{c^{-1}}(u) ) = \Cone_c(v)\cap \left( -\Cone_{c^{-1}}(u)  \cap \alpha_s^\perp \right).
\end{equation}
By Lemma~\ref{DualPrettyLemma}, $-\Cone_{c^{-1}}(u) \cap\alpha_s^\perp$ is a face of $-\Cone_{c^{-1}}( s \join (su)_{\br{s}})$. 
We first argue as in Case 2, replacing $(v,u)$ by $(v, s \join (su)_{\br{s}})$, and conclude, using induction on rank, that $\Cone_c(v)$ and $- \Cone_{c^{-1}}(s \join (su)_{\br{s}})$ meet nicely.
By Lemma~\ref{DualPrettyLemma}, $-\Cone_{c^{-1}}(u)\cap\alpha_s^\perp$ is a face of $-\Cone_{c^{-1}}(s\join(su)_{\br{s}})$, so $\Cone_c(v)$ and $-\Cone_{c^{-1}}(u)\cap\alpha_s^\perp$ meet nicely.  
By \eqref{case4eq}, $\Cone_c(v)\cap(-\Cone_{c^{-1}}(u))$ is a face of $\Cone_c(v)$ and of $-\Cone_{c^{-1}}(u)\cap\alpha_s^\perp$.
Since $-\Cone_{c^{-1}}(u)\cap\alpha_s^\perp$ is a face of $-\Cone_{c^{-1}}(u)$, we see that $\Cone_c(v)\cap(-\Cone_{c^{-1}}(u))$ is a face of $-\Cone_{c^{-1}}(u)$ as well.
\end{proof}

\begin{remark}
There is an alternate strategy for proving Theorem~\ref{DoubleFan}.
The fact that the $\g$-vector fan is indeed a fan is one of the structural conjectures mentioned in the introduction.
It was proved in the acyclic case (and beyond) in \cite{Demonet} and in general in \cite{GHKK}.
As a consequence of Theorem~\ref{framework exchange}\eqref{g vec}, the Cambrian fan $\F_c$ is a subfan of the $\g$-vector fan.
Thus, to prove Theorem~\ref{DoubleFan}, it would be enough to show that $-\F_{c^{-1}}$ is also a subfan of the $\g$-vector fan.
This proof (or an even simpler approach using the mutation fan of \cite{universal} in place of the $\g$-vector fan) can indeed be carried out.
However, here we prefer to construct the doubled Cambrian fan completely in terms of the combinatorics and geometry of Coxeter groups and root systems, without relying on any other heavy machinery.
\end{remark}

The same cone may occur in both $\F_{c}$ and $-\F_{c^{-1}}$.
In fact, 
since the Cambrian fan covers the Tits cone, we have the following immediate corollary to Theorem~\ref{DoubleFan}.

\begin{cor} \label{DoubledConeExists} 
For every $c$-sortable element $v$ such that the interior of $\Cone_c(v)$ meets $-\Tits(A)$, there is a $c^{-1}$-sortable element $v'$ with $\Cone_c(v) = - \Cone_{c^{-1}}(v')$. 
\end{cor}

\begin{remark}\label{double finite}
When $W$ is finite, Corollary~\ref{DoubledConeExists} becomes the statement that $\F_c$ and $-\F_{c^{-1}}$ are identical.
This is essentially \cite[Proposition~1.3]{sort_camb}, although that result is phrased in terms of Cambrian congruences.
In that case, $v' = \pidown^{c^{-1}}(v w_0)$.
\end{remark}

We pause to record some additional useful facts, which reproduce the finite case of \cite[Theorem~9.6]{typefree} and \cite[Proposition~9.7]{typefree} (cf.\ \cite[Proposition~7.4]{camb_fan}) and strengthen the general case of the same theorem and proposition.
These are not needed for the remainder of the paper, but add insight into the structure of $\DF_c$.
For $s$ in $S$, define $\DF_c^{\ab(s)}$ to be the set of cones in $\DF_c$ that are above $\alpha_s^\perp$ and define $\DF_c^{\bel(s)}$ to be the set of cones in $\DF_c$ below $\alpha_s^\perp$.
Suppose $J\subseteq S$ and recall that the restriction of $c$ to $W_J$ is the Coxeter element $c'$ of $W_J$ obtained by deleting the letters in $S\setminus J$ from a reduced word for $c$.
Let $A_J$ be the Cartan matrix obtained by deleting from $A$ the rows and columns indexed by $S\setminus J$, and let $\Phi_J$ be the corresponding sub root system of $\Phi$.
Let $V_J$ be the subspace of $V$ spanned by $\set{\alpha_s:s \in J}$. 
The parabolic subgroup $W_J$ fixes $V_J$ as a set.
Let $\Proj_J$ stand for the surjection from $V^*$ onto $(V_J)^*$ that is dual to the inclusion of $V_J$ into $V$.
Recall that $\br{s}$ means $S\setminus\set{s}$.
When $s$ is initial in $c$, we write $\DF_{sc}$ for the doubled Cambrian fan constructed in $V_\br{s}$ using the Coxeter element $sc$ of $W_\br{s}$.

\begin{prop}\label{recursive fan}
Let $s$ be initial in $c$.
Then    
\begin{enumerate}
\item $\DF_c=\DF_c^{\ab(s)}\cup\DF_c^{\bel(s)}$.
\item $\DF_c^{\ab(s)}=s\bigl[\DF_{scs}^{\bel(s)}\bigr]$.
\item The map $F\mapsto \Proj_{\br{s}}^{-1}(F)\cap\set{x\in V^*: \br{x,\alpha_s}\ge 0}$ is a bijection from the set of maximal cones of $\DF_{sc}$ to the set of maximal cones of $\DF_c^{\bel(s)}$.
\item The map $F\mapsto \Proj_{\br{s}}^{-1}(F)\cap\set{x\in V^*: \br{x,\alpha_s}\le 0}$ is a bijection from the set of maximal cones of $\DF_{sc}$ to the set of maximal cones of $\DF_{scs}^{\ab(s)}$.
\end{enumerate}
\end{prop}
\begin{proof} 
Assertion (1) is the statement that every cone of $\DF_c$ is either above or below $\alpha_s^\perp$.
This is true, independent of the hypothesis that $s$ is initial, by Lemma~\ref{AboveBelow} applied to $\F_c$ and to $-\F_{c^{-1}}$.
Theorem~\ref{c to scs} and equation \eqref{C recur} show that the reflection $s$ takes the cones of $\F_c$ above $\alpha_s^\perp$ bijectively to the cones of $\F_{scs}$ below $\alpha_s^\perp$.
The same theorem and equation show that $s$ takes the cones of $-\F_{c^{-1}}$ above $\alpha_s^\perp$ bijectively to the cones of $-\F_{sc^{-1}s}$ below $\alpha_s^\perp$.
Thus (2) holds. 
Assertion~(3) follows from \eqref{C recur} for cones of $\F_c$ below $\alpha_s^\perp$ and from Proposition~\ref{Cc final} for cones of $-\F_{c^{-1}}$ below $\alpha_s^\perp$.
Assertion (4) holds similarly with the roles of \eqref{C recur} and Proposition~\ref{Cc final} reversed.
\end{proof}

\subsection{Doubling the Cambrian framework}\label{DCambc sec}
Theorem~\ref{DoubleFan} is the key to constructing the doubled Cambrian reflection framework $(\DCamb_c,\DC_c)$.
We define $\DCamb_c$ to be the dual graph to $\DF_c$.
More directly, let $-\Camb_{c^{-1}}$ be $\Camb_{c^{-1}}$ realized as a graph on formal negations $-v$ of $c^{-1}$-sortable elements $v$ with an edge $(-u,-v)$ in $-\Camb_{c^{-1}}$ if and only if $(u,v)$ is an edge in $\Camb_{c^{-1}}$.
Since there is typically no abelian group structure on $c$-sortable elements, this notation should cause no confusion. 
We start with a disjoint union of $\Camb_c$ and $-\Camb_{c^{-1}}$ with labels $C_c(v)$ on each vertex $v$ of $\Camb_c$ and labels $-C_{c^{-1}}(v)$ on each vertex $-v$ of $-\Camb_c$.
We identify a vertex $v$ of $\Camb_c$ with a vertex $-v'$ of $\Camb_{c^{-1}}$ if $\Cone_c(v)=-\Cone_{c^{-1}}(v')$.
In this case, since $\Cone_c(v)$ and $\Cone_{c^{-1}}(v')$ are both defined by roots in the root system associated to $A$, we must have $C_c(v)=-C_{c^{-1}}(v')$.
We think of $C_c$ as a labeling of incident pairs as in Section~\ref{camb sec}, and use the equality $C_c(v)=-C_{c^{-1}}(v')$ of sets to identify the sets $I(v)$ and $I(-v')$ of edges
in the obvious way.
If $v$ has a full edge connecting $v$ to $u$ but the corresponding edge of $-v'$ is a half-edge, then the identified edge is a full edge connecting the identified vertex $v=-v'$ to the vertex $u$, and similarly if $v$ has a half-edge where $-v'$ has a full edge.

This way of defining $\DCamb_c$ also yields a labeling $\DC_c$ of incident pairs of $\DCamb_c$.
For an incident pair $(\tilde v,\tilde e)$ in $\DCamb_c$ corresponding to an incident pair $(v,e)$ in $\Camb_c$, we set $\DC_c(\tilde v,\tilde e)=C_c(v,e)$.
For an incident pair $(\tilde v,\tilde e)$ in $\DCamb_c$ corresponding to an incident pair $(-v,-e)$ in $-\Camb_{c^{-1}}$, we set $\DC_c(\tilde v,\tilde e)=-C_{c^{-1}}(v,e)$.
If a given incident pair in $\DCamb_c$ takes both of these forms, then, by construction, the two ways of computing $\DC_c$ on that incident pair agree.

\begin{example}\label{affineG2 frame ex} 
This example continues Example~\ref{affineG2 ex}.
Figure~\ref{G2tilde frame} shows the graph $\DCamb_c$.
The vertices and edges that exist both in $\Camb_c$ and $-\Camb_{c^{-1}}$ are shown in red.
These correspond to the triangles in Figure~\ref{G2tildeCamb both alt} whose interiors intersect the equator (the red line).
\begin{figure}[ht]
\scalebox{0.87}{\includegraphics{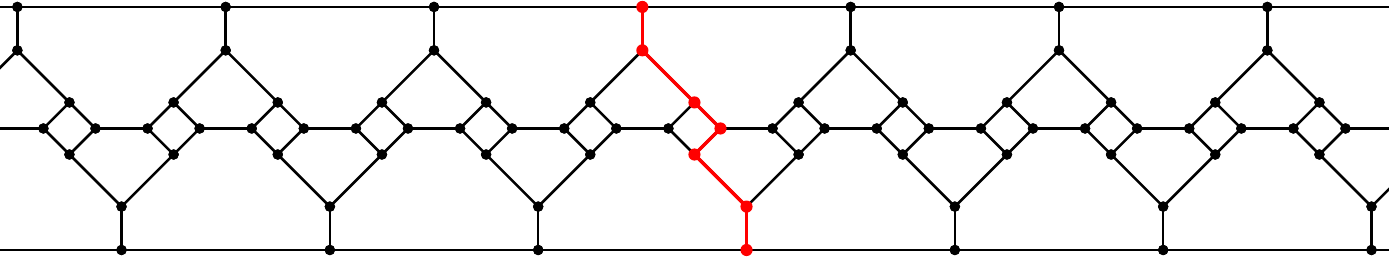}}
\caption{$\DCamb_c$}
\label{G2tilde frame}
\end{figure}
\end{example}

\begin{theorem} \label{DoubleFramework}
If $\DCamb_c$ is connected, then $(\DCamb_c, \DC_c)$ is a polyhedral reflection framework. 
\end{theorem}

\begin{proof}
The Base condition holds in $\DCamb_c$ because it holds in $\Camb_c$.
The Root condition and the Euler conditions hold in $\DCamb_c$ because they hold in $\Camb_c$ and $\Camb_{c^{-1}}$, because $E_{c^{-1}}(\beta, \gamma) = E_c(\gamma, \beta)$, and because passing from $\Camb_c$ to $-\Camb_c$ corresponds to changing the signs of all labels.
Finally, each full edge~$e$ in $\DCamb_c$ comes from a full edge in $\Camb_c$ or a full edge in $-\Camb_{\c^{-1}}$, or both.
If~$e$ comes from a full edge in $\Camb_c$, then the Reflection condition holds in $\DCamb_c$ because it holds in $\Camb_c$. 
If $e$ comes from a full edge in $\Camb_{c^{-1}}$, then the Reflection condition holds in $\Camb_{c^{-1}}$.
Passing from the Reflection condition in $\Camb_{c^{-1}}$ to the Reflection condition in $\DCamb_c$, the root $\beta_t$ does not change sign because it is positive by definition, but $\gamma$ does change sign because we pass from label sets $C_{c^{-1}}$ 
in $\Camb_{c^{-1}}$ to label sets $-C_{c^{-1}}$ in $\DCamb_c$.
However, $\omega_{c}(\beta_t, -\gamma) = \omega_{c^{-1}}(\beta_t, \gamma)$, so the Reflection condition is preserved.
Thus $\DCamb_c$ satisfies the definition of a reflection framework. 
By construction, $\DF_c$ is the collection of cones associated to vertices of $\DCamb_c$.
By Theorem~\ref{DoubleFan}, $\DF_c$ is a fan.
Also by construction, if a cone of $\F_c$ and a cone of $-\F_{c^{-1}}$ coincide, then the corresponding vertices are identified in $\DCamb_c$.
In light of Theorem~\ref{camb good}, we see that distinct vertices of $\DCamb_c$ label distinct maximal cones of $\DF_c$.
Thus $(\DCamb_c, \DC_c)$ is polyhedral.
\end{proof}

The hypothesis of Theorem~\ref{DoubleFramework} only holds when $\F_c$ and $-\F_{c^{-1}}$ have full-dimensional cones in common.  
The following proposition tells us that this happens in many cases.
As special cases of the proposition, $\DCamb_c$ is connected in the affine case and, more generally, whenever every proper parabolic subgroup of $W$ is finite.

\begin{prop} \label{OverlapExists}  
Suppose $s_1\cdots s_n$ is a reduced word for a Coxeter element $c$ of $W$, and suppose that, for some $i\in\set{1,\ldots,n-1}$, the parabolic subgroups $W_\set{s_1,\ldots,s_i}$ and $W_\set{s_{i+1},\ldots,s_n}$ are both finite.
Then $\DCamb_c$ is connected.
\end{prop}

We first need the following lemma. We write $(w_0)_J$ for the longest element of $W_J$, when $W_J$ is finite.

\begin{lemma} \label{w0Region}
Let $c = s_1 s_2 \cdots s_n$ and let $J=\set{s_{m+1},\ldots,s_n}$ for some $m\ge1$. 
Suppose that $W_J$  
 is finite. 
Then $C_c((w_0)_J)=\set{\alpha_1,\alpha_2,\ldots,\alpha_m,-\alpha_{m+1},\ldots,-\alpha_n}$.
\end{lemma}

\begin{proof}  
We argue by induction on $n-m$. 
If $n-m=0$, so that $(w_0)_J$ is the identity, this is immediate. 
If $n-m>0$, then let $s$ stand for $s_n$ and write $J'=\set{s_{m+1},\ldots,s_{n-1}}$.
Then $C_{cs}((w_0)_{J'})=\set{\alpha_1,\alpha_2,\ldots,\alpha_m,-\alpha_{m+1},\ldots,-\alpha_{n-1}}$ by induction.  
Since $(w_0)_J\ge s$ and $[(w_0)_J]_\br{s}=(w_0)_{J'}$, we apply Proposition~\ref{Cc final} to conclude that $C_c((w_0)_J)=C_{cs}((w_0)_{J'})\cup\set{-\alpha_s}$.
\end{proof}

\begin{proof}[Proof of Proposition~\ref{OverlapExists}]  
Write $J=\set{s_{i+1},\ldots,s_n}$.
The element $(w_0)_J$ is $c$-sortable by Theorem~\ref{meets and joins} because it is the join of $J$.
Similarly, $(w_0)_{S\setminus J}$ is $c^{-1}$-sortable.
Lemma~\ref{w0Region} implies that
\[C_c((w_0)_J) = \{ \alpha_1, \ldots,\alpha_i,-\alpha_{i+1}, \ldots, - \alpha_n \} = - C_{c^{-1}}((w_0)_{S\setminus J}).\] 
We have identified a cone that is in $\F_c$ and in $-\F_{c^{-1}}$.
\end{proof}

\begin{remark}\label{rank 2}
If $B$ is $2\times2$, then by Proposition~\ref{OverlapExists} (or by an easy direct check), $\DCamb_c$ is connected.
Thus $(\DCamb_c, \DC_c)$ is a polyhedral reflection framework by Theorem~\ref{DoubleFramework}.
One can also easily check that $(\DCamb_c, \DC_c)$ is complete, exact, simply connected, and 
 well-connected.
\end{remark}

\begin{remark} \label{double green remark}  
The identity element is $c^{-1}$-sortable for any $c$, so the formal negative of the identity is a vertex of $\DCamb_c$.
The labels on the negative of the identity are the negative simple roots.
Consider a path $v_0,\ldots,v_p$ in $\DCamb_c$ such that $v_0$ is the identity and write $e_i$ for the edge connecting $v_{i-1}$ to $v_i$.
As pointed out in \cite[Remark~4.19]{framework} 
for general frameworks, 
 such a path is a \newword{green sequence} in the sense of~\cite{MGS} (see also~\cite{MGS2}) if and only if all of the roots $\DC_c(v_i, e_i)$ are positive.
We also observe that this green sequence is a \newword{maximal green sequence} if and only if $v_p$ is the negative of the identity.
The following proposition is worth mentioning, even though it does not add to the list of exchange matrices $B$ for which maximal green sequences are known to exist.
(Maximal green sequences are known \mbox{\cite[Lemma~2.20]{MGS2}} to exist for all acyclic $B$.)  
\end{remark}

\begin{prop}\label{double green prop}
If $\DCamb_c$ is connected, then it contains a maximal green sequence.
\end{prop}
\begin{proof}
If $\DCamb_c$ is connected, then there exists a $c$-sortable element $u$ and a $c^{-1}$-sortable element $v$ such that $\Cone_c(u)=-\Cone_{c^{-1}}(v)$.
Lemma~\ref{cov beta part 1} and the definition of the Cambrian framework $(\Camb_c,C_c)$ imply that any unrefinable chain $u_0,\cdots,u_k$ from the identity to $u$ is a green sequence in $\Camb_c$.
Similarly, any unrefinable chain $v_0,\cdots,v_\ell$ from the identity to $v$ is a green sequence in $\Camb_{c^{-1}}$.
Using the symbols $u_i$ and $-v_i$ to label the corresponding vertices in $\DCamb_c$, the sequence $u_0,\ldots,u_k=-v_\ell,\ldots,-v_0$ is a maximal green sequence.
\end{proof}

In the next section, we establish some additional properties of the doubled Cambrian framework when $B$ is acyclic and $\Cart(B)$ is of affine type,
including completeness.
The Cartan companion in Examples~\ref{affineG2 ex} and~\ref{affineG2 frame ex} is of affine type, 
and the completeness of the doubled Cambrian framework is seen in the fact that the dual graph to $\DF_c$ depicted in Figure~\ref{G2tilde frame} is $3$-regular.
Before we confine our attention to the affine case, we give an example where the doubled Cambrian framework is not complete.

\begin{figure}[ht]
\includegraphics{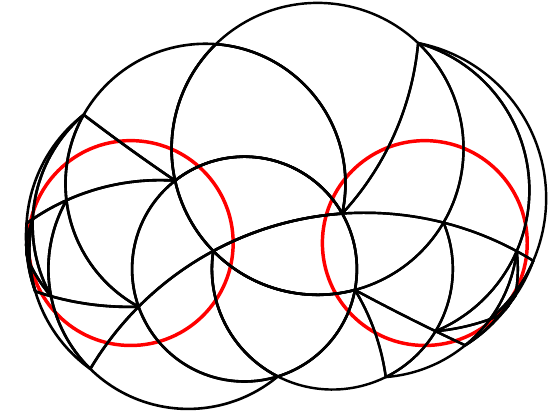}
\caption{A fan $\DF_c$ whose corresponding framework is not complete}
\label{344 fig}
\end{figure}
\begin{example}\label{344 ex}
Figure~\ref{344 fig} shows the doubled Cambrian fan for $B=\begin{bsmallmatrix*}[r]0&-1&-2\\1&0&-2\\1&1&0 \end{bsmallmatrix*}$.
The two red circles show the boundary of the Tits cone and the boundary of the negative Tits cone.
The Cambrian fan $\F_c$ consists of the cones whose projections (curvilinear triangles) intersect the inside of the left circle.
The fan $-\F_{c^{-1}}$ consists of cones whose projections intersect the inside of the right circle.
The dual graph $\DCamb_c$ to $\DF_c$ is not $3$-regular, so $(\DCamb_c, DC_c)$ is not a complete framework in this case.
Specifically, the exterior of the figure is not a cone of $\DF_c$, so the cones that border the exterior correspond to vertices of degree less than $3$.
However, the dual graph 
is connected because there are cones that meet both the interior of $\Tits(A)$ and the interior of $-\Tits(A)$. 
\end{example}

\section{Affine type}\label{affine sec}  
In this section, we discuss root systems of affine type, 
gather some useful facts about them, and complete the proof of Theorem~\ref{AffineDoubleFramework}.
This theorem says that when $B$ is acyclic and its Cartan companion $\Cart(B)$ is of affine type, 
$(\DCamb_c, \DC_c)$ is a complete, exact, well-connected, polyhedral, simply connected reflection framework.  

By Theorem~\ref{DoubleFramework}, if $(\DCamb_c, \DC_c)$ is connected, then it is a polyhedral reflection framework. 
Recall that the polyhedral property implies injectivity, that simple connectivity implies ampleness, and that exact means injective and ample.
Thus the task is to show connectedness, completeness, well-connectedness, and simple connectivity.
We do this in Propositions~\ref{connected}, \ref{complete}, \ref{well-connected}, and~\ref{simply}.

\subsection{Affine type and affine root systems} \label{affine root sec} 
A good reference on affine root systems is~\cite[Chapters 4-6]{Kac}.
See also~\cite{Macdonald}.  

As before, let $B$ be a skew-symmetrizable matrix, with Cartan companion $A=\Cart(B)$, 
with root system $\Phi$ and symmetric bilinear form $K$. 
The Cartan matrix $A$ is of \newword{affine type} if $K$ is positive semidefinite, $K$ is not positive definite and, for any proper subset $J \subsetneq S$, the restriction of $K$ to $V_J$ is positive definite.   
This definition of affine type is equivalent to that of~\cite{Kac}. 
The equivalence is~\cite[Proposition 4.7.b]{Kac}, up to a technical point: 
Whereas~\cite{Kac} explicitly requires that affine Cartan matrices are indecomposable, our definition (and in particular the positive definiteness of restrictions to $V_J$) easily implies indecomposability.
When $A$ is of affine type, we say that the Coxeter group $W$ and the root system $\Phi$ are affine.  

The affine Coxeter groups are in bijection with the indecomposable finite (crystallographic) root systems.  
Specifically, if $W$ is an affine Coxeter group, then there is an $(n-1)$-element subset $S_0$ of $S$, and an indecomposable finite root system $\Phi_0$ with Coxeter group $W_0$, such that $W$ is isomorphic to a semidirect product of $W_0$ with the lattice generated by the co-roots $\Phi_0\ck$.
The subset $S_0$ may not be uniquely determined, but we will fix a choice of $S_0$.
Write $s_{\aff}$ for the unique element of $S \setminus S_0$.
We say that $W$ has type $\tilde{X}$, where $X$ is the type of $\Phi_0$. 
(The type of $\Phi_0$ does not depend on the choice of $S_0$.)

Since our goal is to construct a framework for the exchange matrix $B$, we must talk about roots.
This raises an issue which is sometimes overlooked in discussions of affine Coxeter groups: 
Just as the finite Coxeter group of type $B_n$ comes from two different root systems ($B_n$ and $C_n$),
 there are affine Coxeter groups that come from more than one root system.
For example, the Cartan matrices $\begin{bsmallmatrix*}[r] 2 & -1 \\ -4 & 2 \end{bsmallmatrix*}$ and $\begin{bsmallmatrix*}[r] 2 & -2 \\ -2 & 2 \end{bsmallmatrix*}$ both specify the Coxeter group of type $\tilde{A}_1$.
The root systems for these two Cartan matrices are different and this difference can be seen when we construct frameworks for corresponding exchange matrices. 

We now review the theory of affine root systems.
Let $\Phi$ be an affine (real) root system associated to a Cartan matrix $A$.
The associated symmetric bilinear form $K$ has a one-dimensional kernel.
The nonzero 
elements of the root lattice contained in the kernel are called the \newword{imaginary roots}, and like the real roots, these consist of two classes, the \newword{positive imaginary roots} contained in the nonnegative span of the simple roots $\Pi$ and the \newword{negative imaginary roots} contained in the nonpositive span of $\Pi$.
Let $\delta$ be the shortest positive imaginary root (the nonzero 
element of the root lattice closest to the origin in the ray spanned by positive imaginary roots).
The vector $\delta$ has strictly positive coordinates with respect to the basis $\Pi$ of simple roots.
Since $\delta$ is in the kernel of $K$, it is fixed by the action of $W$.

Corresponding to the decomposition $S=S_0\sqcup\set{s_\aff}$ of the simple reflections in affine \emph{Coxeter groups} is a decomposition $\Pi=\Pi_0\sqcup\set{\alpha_\aff}$ with the following properties:
The subset $\Pi_0$ spans an indecomposable finite root system $\Phi_0$;
and the simple root $\alpha_\aff$ is a positive scaling of $\delta-\theta$, where $\theta$ is a positive root in $\Phi_0$.
(The root $\theta$ is either the highest root of $\Phi_0$ or the highest short root of $\Phi_0$.)
Every root in $\Phi$ is a positive scaling of a vector of the form $\beta+k\delta$, where $\beta\in\Phi_0$ and $k\in\integers$.
A root obtained as a positive scaling of $\beta+k\delta$ is a positive root if and only if either $k$ is positive or $k=0$ and $\beta$ is positive.
In some affine root systems $\Phi$, there exist vectors $\beta+k\delta$ with $\beta\in\Phi_0$ and $k\in\integers$ that are not scalings of roots.
However, for each $\beta\in\Phi_0$, there are infinitely many positive integers $k$ and infinitely many negative integers $k$ such that $\beta+k\delta$ is a positive scaling of a root.
See \cite[Proposition~6.3]{Kac} for more details.

Let $W$ and $W_0$ be the Coxeter groups associated to $\Phi$ and $\Phi_0$ respectively, let $V$ and $V_0$ be their respective reflection representations, with $V_0\subset V$, and let $V^*$ and $V_0^*$ be the respective dual representations.
Thus $V = V_0 \oplus \RR \delta$.
We write $\pi$ for the projection onto $V_0$ with kernel $\reals\delta$.
The Tits cone $\Tits(A)$ is ${\set{x\in V^*:\br{x,\delta} >0}\cup\set{0}}$.  
The boundary $\partial\Tits(A)$ of the Tits cone is $\delta^{\perp}$. 
The inclusion of $V_0^*$ into $V^*$, dual to the projection $\pi$, identifies $\partial\Tits(A)$ with $V_0^*$.

The action of $W$ preserves the set $V_1^*:=\set{x\in V^*:\br{x,\delta}=1}$. 
The action of $W$ on $V_1^*$ is generated by reflections in affine hyperplanes $\set{x\in V_1^*:\br{x,\beta}=k}$ in $V_1^*$ for $k\in\ZZ$ and $\beta\in\Phi_0$.
Thus 
we can view $W$ as a group of affine Euclidean motions of an $(n-1)$-dimensional vector space. 
In this context, we define $D_1$ to be $D \cap V_1^*$ where, as before, $D$ is the cone $\bigcap_{s \in S} \set{x\in V^*: \br{x,\alpha_s}\ge 0}$.
Then $V_1^{\ast}$ is tiled by translates of the simplex $D_1$.
For an element $w$ of $W$, the inversions of $w$ correspond to the affine hyperplanes separating $w D_1$ from $D_1$.
For a presentation from this perspective, see~\cite[Section VI.2]{Bour}.

The best-known affine root systems are the ``standard" affine root systems, in bijection with, and constructed directly from, the finite root systems.
The Dynkin diagrams of these root systems are shown in Table Aff~1 of \cite[Chapter~4]{Kac}.
In the standard affine root systems, the roots are exactly the vectors of the form $\beta+k\delta$ and the reflecting hyperplanes in $V^*$ are precisely the hyperplanes of the form $\br{\beta,x}=k\br{\delta,x}$ for $k\in\ZZ$ and $\beta\in \Phi_0$.
Every affine root system can be obtained from a standard affine root system
 by rescaling roots (necessarily with consistent rescaling within $W$-orbits of roots) and leaving the underlying bilinear form $K$ unchanged.
Rescaling the root system does not change the Coxeter group or the reflection representation, so the reader may ignore this issue in any discussion which refers only to these concepts.

\begin{example}\label{nonstandard}
As an example of a ``nonstandard'' affine root system, take $A =\begin{bsmallmatrix*}[r] 2 & -1 \\ -4 & 2 \end{bsmallmatrix*}$ and construct a root system $\Phi$ with simple roots $\alpha_1$ and $\alpha_2$ and simple co-roots $\alpha_1\ck=\alpha_1$ and $\alpha_2\ck=4\alpha_2$.
Set $\Phi_0=\set{\pm \alpha_1}$, so that $S_0=\set{s_1}$ and ${\alpha_\aff=\alpha_2}$.  
The kernel of $A$ is the real span of $\delta=\alpha_1+2\alpha_2$.
The highest root of $\Phi_0$ is $\theta=\alpha_1$, so $\alpha_\aff=\alpha_2=\frac12(\delta-\theta)$.
The roots in the $W$-orbit of $\alpha_1$ are of the form $\pm\alpha_1+2k\delta$ for integers $k$, while roots in the $W$-orbit of $\alpha_2$ are $\frac{1}{2}(\pm\alpha_1+(2k+1)\delta)$ for integers~$k$.

This root system can be rescaled to a standard affine root system with simple roots $\alpha'_1=\alpha_1$ and $\alpha'_2=2\alpha_2$.
We have $K(\alpha_1',\alpha_1')=K(\alpha_1,\alpha_1)=2$ and $K(\alpha_2',\alpha_2')=K(2\alpha_2,2\alpha_2)=2$, so $(\alpha'_1)\ck=\alpha_1'$, and $(\alpha_2')\ck=\alpha_2'$.
The Cartan matrix $A'$ for $\Phi'$ is $A' =\begin{bsmallmatrix*}[r] 2 & -2 \\ -2 & 2 \end{bsmallmatrix*}$.
The vector $\delta$ is $\alpha'_1+\alpha_2'$, and the roots in $\Phi'$ are $\pm\alpha_1+k\delta$ for integers $k$.
\end{example}

One more technical observation will be useful.
Write $t_\theta$ for the reflection with respect to the root $\theta$.
Since $\theta$ is a root in $\Phi_0\subset\Phi$, the reflection $t_\theta$ makes sense in both reflection representations $V$ and $V_0$.
Since $\alpha_\aff$ is a positive scaling of $\delta-\theta$ and since $K(\delta,x)=0$ for all $x\in V$, we can rewrite the action of $s_\aff$ on $V$ in terms of $t_\theta$ and $\delta$ as follows. 
\begin{equation}\label{simplify reflection}
s_\aff(x)=t_\theta(x)+\br{\theta\ck,x}\delta.
\end{equation}

\subsection{Connectedness and completeness}\label{con com sec}
We now establish the properties of connectedness and completeness for $(\DCamb_c,\DC_c)$.

\begin{prop}\label{connected}
If $A$ is of affine type, then $\DCamb_c$ is connected.
\end{prop}

\begin{proof}
When $\Cart(B)$ is of affine type, the corresponding Coxeter group $W$ is of affine type. 
Thus every proper parabolic subgroup of $W$ is of finite type, so the result follows from Proposition~\ref{OverlapExists}.
\end{proof}

\begin{prop}\label{complete}
If $A$ is of affine type, then $(\DCamb_c,\DC_c)$ is complete.
\end{prop}

\begin{proof}
The assertion is that $\DCamb_c$ has no half-edges.
Given a vertex $\tilde v$ in $\DCamb_c$, we may as well (up to swapping $c$ and $c^{-1}$ and negating all labels) 
assume that the vertex is represented by a $c$-sortable element $v$, i.e.\ a vertex of $\Camb_c$.
Let~$\tilde e$ be an edge incident to $\tilde v$ in $\DCamb_c$, and let $e$ be the corresponding edge in $\Camb_c$.
If $C_c(v,e)$ is a negative root, then by the Full edge condition on $\Camb_c$, the edge $e$ is a full edge, so $\tilde e$ is a full edge.
Suppose $C_c(v,e)$ is positive.
If the interior of $\Cone_c(v)$ intersects $-\Tits(A)$, then Theorem~\ref{DoubledConeExists} says that there exists a $c^{-1}$-sortable element 
 $v'$ with $\Cone_c(v) = - \Cone_{c^{-1}}(v')$. 
The edge $\tilde e$ corresponds to an edge $e'$ of $-\Camb_{c^{-1}}$, which must be a full edge, by the Full edge condition, since $C_{c^{-1}}(v',e')=-C_c(v,e)$ is negative.
Thus $\tilde e$ is a full edge.
If the interior of $\Cone_c(v)$ is disjoint from $-\Tits(A)$, then since the closure of $\Tits(A)\cup(-\Tits(A))$ is $V$, the interior of $\Cone_c(v)$ is contained in $\Tits(A)$.
Thus, since $C_c(v,e)$ is a positive root, there are elements $w_1\covered w_2$ in $W$ such that $w_1 D\in\Cone_c(v)$ and $w_2 D\not\in\Cone_c(v)$, with the cones $w_1 D$ and $w_2 D$ separated by $C_c(v,e)^{\perp}$.
Then $v=\pidown^c(w_1)$, and we define $u=\pidown^c(w_2)$.
Since $\F_c$ is a fan (Theorem~\ref{camb fan}) and because $C(v,e)$ is an inward-facing normal to a facet of $\Cone_c(v)$, the vector $-C_c(v,e)$ is an inward facing normal to a facet of $\Cone_c(u)$.
Thus $-C_c(v,e)$ is $C_c(u,f)$ for some $f\in I(u)$. 
The Full edge condition says that $f$ connects $u$ to a vertex $w$, and thus $\Cone_c(u)$ and $\Cone_c(w)$ share a facet defined by $C_c(u,f)$.
Since $\F_c$ is a fan, we must have $w=v$ and $f=e$.
In particular, $\tilde e$ is a full edge.
\end{proof}

As consequences of Proposition~\ref{complete}, we obtain some important facts about $\DF_c$ when $\Cart(B)$ is of affine type.
First, since the maximal cones of $\DF_c$ are simplicial and $n$-dimensional, every maximal cone in $\DF_c$ is adjacent to exactly $n$ other maximal cones.
Also, every codimension-$1$ cone in $\DF_c$ is contained in exactly two full-dimensional cones.

\subsection{Faces of $\DF_c$ and the boundary of $\Tits(A)$}\label{faces boundary sec}
We now prove some key facts about how faces of the affine doubled Cambrian fan intersect the boundary of the Tits cone.  
We will use repeatedly the fact that each cone of $\DF_c$ is the intersection of half-spaces of the form $\set{x\in V^* : \br{x,\beta} \geq 0}$ for $\beta \in \Phi$. 

\begin{prop}\label{face in boundary}  
If $A$ is of affine type, then every face of $\DF_c$ contained in $\partial\Tits(A)$ has dimension less than $n-1$. 
All $n-1$ dimensional cones in $\DF_c \cap \partial \Tits(A)$ are of the form $F \cap \partial \Tits(A)$ where $F$ is an $n$-dimensional cone of $\DF_c$ and $\partial \Tits(A)$ passes through the interior of $F$.
\end{prop}

\begin{proof}
Every $(n-1)$-dimensional cone of $\DF_c \cap \partial \Tits(A)$ is either (1) of the form $F \cap \partial \Tits(A)$ for $F$ an $n$-dimensional cone of $\DF_c$ whose interior meets $\partial \Tits(A)$ or (2) an $(n-1)$-dimensional cone of $\DF_c$. So the first sentence implies the second.

To see the truth of the first sentence, let $G$ be an $(n-1)$-dimensional cone of $\DF_c$. Then the normal vector to $G$ is a real root. But the normal to $\partial \Tits(A)$ is $\delta$, which is not a scaling of a real root.
\end{proof}

Recall that $\pi$ is the projection of $V$ onto $V_0$ with kernel $\reals\delta$.
The intersection $\{ x\in V^* : \br{x,\beta} \geq 0 \}\cap\partial \Tits(A)$ is ${\{ x\in V_0^* : \br{x,\pi(\beta)} \geq 0 \}}$, and $\pi(\beta)$ is a scaling of a root in $\Phi_0$.
Thus each cone of $\DF_c$ meets $\partial \Tits(A)$ in a cone of $V_0^*$ defined by half spaces corresponding to roots in $\Phi_0$.
In other words, each cone of $\DF_c$ meets $\partial \Tits(A)$ in a cone that is a union of cones in the Coxeter fan of $W_0$.
(This is the fan whose maximal cones are the closures of the connected components of the complement, in $V_0^*$, of the union of the reflecting hyperplanes of $W_0$.)  Since there are only finitely many cones in the Coxeter fan of $W_0$, we obtain the following fact. 

\begin{prop}\label{affine finitely non-maximal}
If $A$ is of affine type, then there are only finitely many different cones of the form $F\cap\partial\Tits(A)$ where $F$ is a cone in $\DF_c$.
\end{prop}

By Proposition~\ref{face in boundary}, the intersection of a maximal cone $F$ with $\partial\Tits(A)$ is $(n-1)$-dimensional if and only if the interior of $F$ intersects $\partial\Tits(A)$.
Since distinct maximal cones of $\DF_c$ have disjoint relative interiors (by Theorem~\ref{DoubleFan}), we also have the following observation.

\begin{prop}\label{affine finitely}
If $A$ is of affine type, then there are only finitely many maximal cones in $\DF_c$ whose interior intersects $\partial\Tits(A)$.  
\end{prop}

Define $D_0 := \bigcap_{s \in S_0} \set{x\in V_0^*: \br{x,\alpha_s}\ge 0}$.
The cone $D_0$ is to $\Phi_0$ as $D$ is to $\Phi$.
The following proposition records a special case of the observation, made above, that each cone of $\DF_c$ meets $\partial \Tits(A)$ in a cone that is a union of cones in the Coxeter fan of $W_0$.

\begin{prop} \label{finite face}  
For $A$ of affine type, if $F$ is a maximal cone of $\DF_c$ with ${F \cap \partial \Tits(A)}$ of dimension $n-1$, then $F \cap \partial \Tits(A)$ is $\bigcup_{u_0} u_0 D_0$, with $u_0$ running over some subset of $W_0$. 
\end{prop}
\subsection{The support of the doubled Cambrian fan} \label{affine fan support} 
The \newword{support} $|\F|$ of a fan  $\F$ is $\bigcup_{F\in\F} F$.
We now 
present a theorem and a corollary 
which describe the support of the doubled Cambrian fan $\DF_c$ in the affine case.  
We define 
\[ \begin{array}{lcl}
\Phi_0^{\omega+} &=& \{ \beta \in \Phi_0 \ : \ \omega_c(\beta, \delta) > 0 \} \\[1pt]
\Phi_0^{\omega0} &=& \{ \beta \in \Phi_0 \ : \ \omega_c(\beta, \delta) = 0 \} \\[1pt]
\Phi_0^{\omega-} &=& \{ \beta \in \Phi_0 \ : \ \omega_c(\beta, \delta) < 0 \}. 
\end{array} \]
If we use the skew symmetric form $\omega$ to define a map $V \to V^{\ast}$, then $\Phi_0^{\omega+}$, $\Phi_0^{\omega0}$ and $\Phi_0^{\omega-}$ are the roots of $\Phi_0$ which are sent to the interior of the Tits cone, boundary of the Tits cone, and interior of the negative Tits cone respectively.  

The map $\gamma\mapsto-\omega_c(\gamma,\delta)$ is a linear functional on $V_0$ and thus a point in $V^*_0$.
We call this point $x_c$.  

With this notation in place, we state the main results of this section.

\begin{theorem} \label{TitsCapBdry}
If $A$ is of affine type, then $|\F_c| \cap \partial \Tits(A) =|-\F_{c^{-1}}| \cap \partial \Tits(A)$, and thus both equal $|\DF_c| \cap \partial \Tits(A)$.  
Furthermore, $|\DF_c| \cap \partial \Tits(A)$ is nonempty and
\[|\DF_c| \cap \partial \Tits(A) = \bigcup_{\beta \in \Phi_0^{\omega +}} \{ x \in \partial \Tits(A) : \br{x,\beta} \geq 0 \}.\]
\end{theorem}

\begin{cor} \label{DFc complement} 
If $A$ is of affine type, then $V^*\setminus|\DF_c|$ is a nonempty open cone in $\partial\Tits(A)$, and thus is of codimension $1$ in $V^*$.
This cone is
\[ V^*\setminus|\DF_c|=\bigcap_{\beta \in \Phi_0^{\omega +}} \{ x \in \partial \Tits(A) : \br{x,\beta} < 0 \}. \]
Also, $V^* \setminus |\DF_c|$ is the union of the relative interiors of those cones in the $W_0$-Coxeter fan that contain $x_c$.  
\end{cor}

\begin{example} \label{affineG2 slice ex}
We illustrate these results by continuing Examples~\ref{affineG2 ex} and~\ref{affineG2 frame ex}, with exchange matrix $B=\begin{bsmallmatrix*}[r]0&1&1\\-3&0&0\\-1&0&0 \end{bsmallmatrix*}$.
The reader may wish to compare Figures~\ref{g2s1}, \ref{g2s-1} and~\ref{g2s0} with Figures~\ref{G2tildeCamb both} and \ref{G2tildeCamb both alt}, which display the same fan in stereographic projection.

The associated Coxeter group is of type $\widetilde{G}_2$. 
The Coxeter element defined by $B$ is $c=s_1s_2s_3$.
The subgroup $W_0$ is generated by $s_1$ and $s_2$.
Figure~\ref{g2root}
shows the root system $\Phi_0$ for $W_0$ and part of the weight lattice, with some weights labeled.
(Recall from Section~\ref{frame sec} that the fundamental weights $\rho_i$ are the dual basis to the simple co-roots.)
In this example, $\delta=2\alpha_1+3\alpha_2+\alpha_3$ and $\Phi_0^{\omega0}=\set{\pm(\alpha_1+2\alpha_2)}$.
The set $\Phi_0^{\omega+}$ is $\set{2\alpha_1+3\alpha_2,\alpha_1+\alpha_2,\alpha_1,-\alpha_2,-\alpha_1-3\alpha_2}$ and $x_c$ is $-4\rho_1+6\rho_2$.
\begin{figure}[ht]
\begin{picture}(0,0)
\end{picture}
\scalebox{0.9}{
\includegraphics{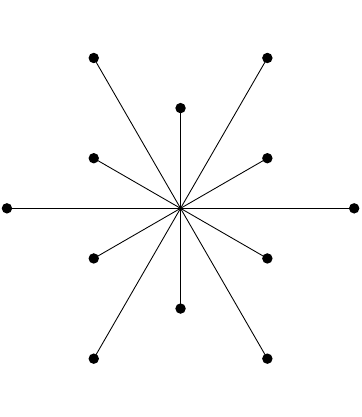}
\begin{picture}(0,0)(52,-60)
\put(-33,18){$\alpha_2$}
\put(43,4.5){$\alpha_1$}
\end{picture}
}
\qquad
\begin{picture}(0,0)
\end{picture}
\scalebox{0.9}{
\includegraphics{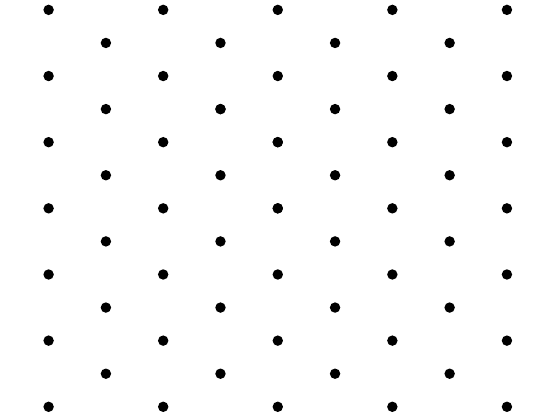}
\begin{picture}(0,0)(80,-60)
\put(-1.5,0){$0$}
\put(-8,23){$\rho_2$}
\put(9,32.5){$\rho_1$}
\put(-73,3.5){$x_c$}

\end{picture}
}
\caption{The root system $\Phi_0$ and weight lattice for Example~\ref{affineG2 slice ex}}
\label{g2root}
\vspace{30 pt}
\end{figure}

\begin{figure}[p]
\begin{picture}(0,0)
\end{picture}
\scalebox{.65}{
\includegraphics{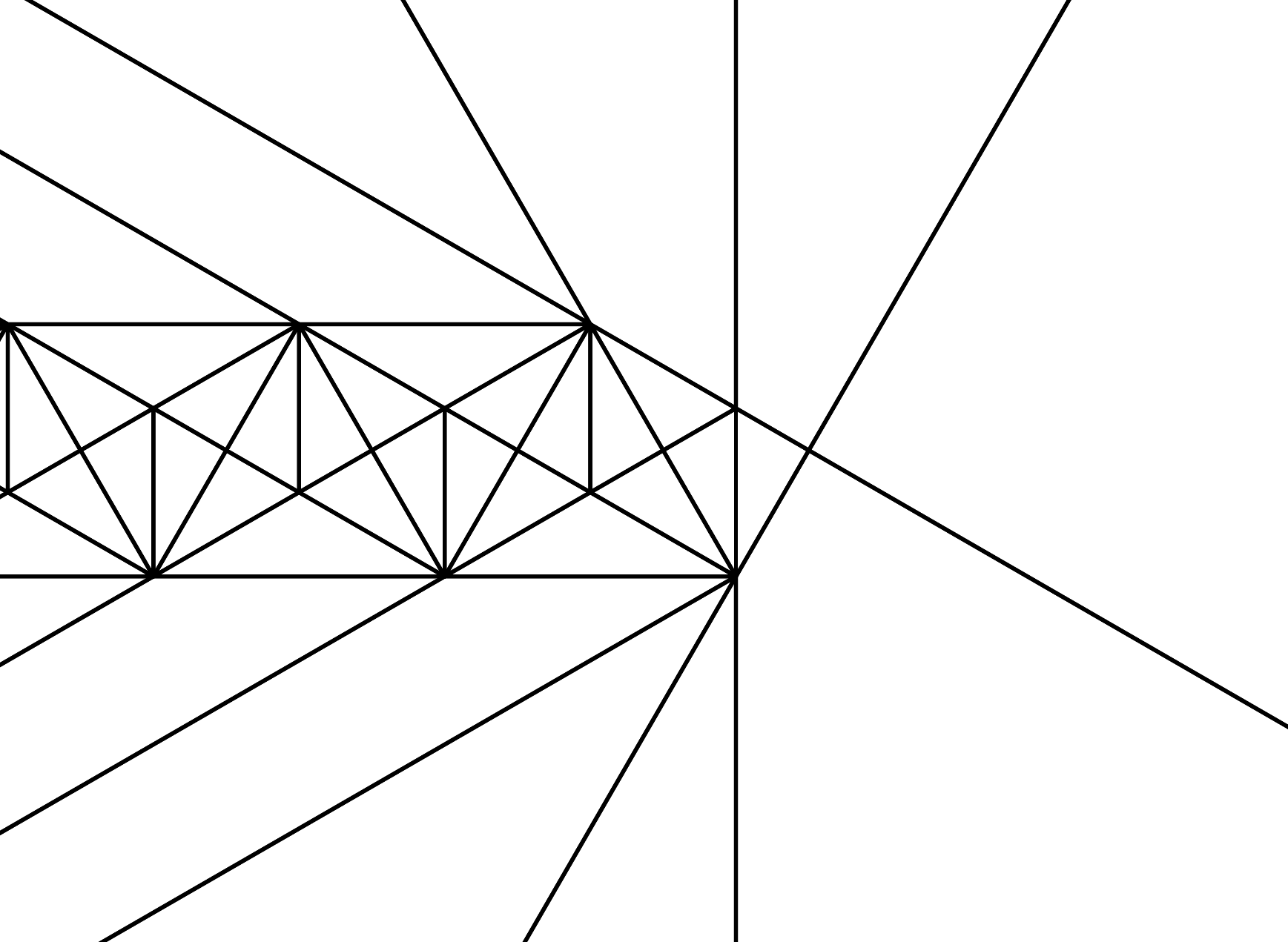}
\begin{picture}(0,0)(200,-191)
\put(-9,-10){\LARGE$e$}
\put(-34,-10){\LARGE$s_1$}
\put(5,60){\LARGE$s_3$}
\put(-53,17){\LARGE$s_1s_3$}
\put(-80,100){\LARGE$s_1s_3s_1$}
\put(-248,160){\LARGE$s_1s_2s_3s_1s_2s_1s_2s_1$}
\put(-63,-19){\LARGE$s_1s_2$}
\put(-92,-38){\LARGE$s_1s_2s_1$}
\put(-135,-75){\LARGE$s_1s_2s_1s_2$}
\put(-161,-155){\LARGE$s_1s_2s_1s_2s_1$}
\put(-93,-186){\LARGE$s_1s_2s_1s_2s_1s_2$}
\put(20,-80){\LARGE$s_2$}
\put(50,10){\LARGE$s_2s_3$}
\end{picture}
}
\caption{The Cambrian fan $\F_c$ of Example~\ref{affineG2 slice ex}}
\label{g2s1}
\vspace{30 pt}
\end{figure}
\begin{figure}[p]
\begin{picture}(0,0)
\end{picture}
\scalebox{.65}{
\includegraphics{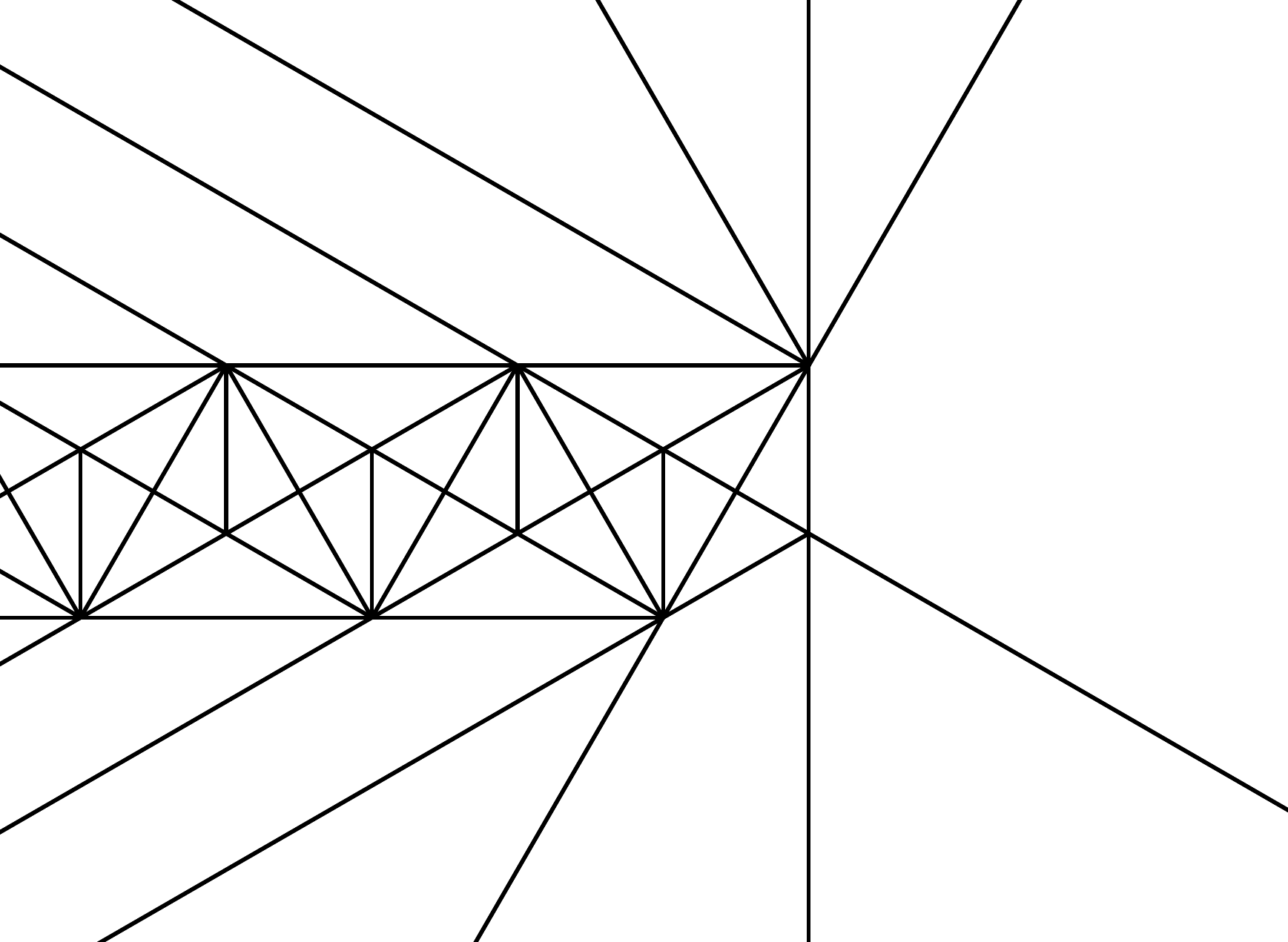}
\begin{picture}(0,0)(200,-175)
\put(-3,5){\LARGE$e$}
\put(50,10){\LARGE$s_1$}
\put(-19,-20){\LARGE$s_3$}
\put(-53,-140){\LARGE$s_3s_1$}
\put(36,-120){\LARGE$s_3s_1s_3$}
\put(-210,-160){\LARGE$s_3s_2s_1s_2s_1s_2s_1$}
\put(-27,15){\LARGE$s_2$}
\put(-57,33){\LARGE$s_2s_1$}
\put(-120,80){\LARGE$s_2s_1s_2$}
\put(-112,140){\LARGE$s_2s_1s_2s_1$}
\put(-54,168){\LARGE$s_2s_1s_2s_1s_2$}
\put(13,186){\LARGE$s_2s_1s_2s_1s_2s_1$}
\end{picture}
}
\caption{The opposite Cambrian fan $-\F_{c^{-1}}$ of Example~\ref{affineG2 slice ex}}
\label{g2s-1}
\end{figure}
Figure~\ref{g2s1} shows $\DF_c$ intersected with the affine plane $V_1^*=\set{x\in V^*:\br{x,\delta}=1}$.
Nonzero cones in $\DF_c$ intersect $V_1^*$ if and only if they intersect the interior of $\Tits(A)$, so Figure~\ref{g2s1} is a representation of $\F_c$.
(Essentially the same picture appears as \cite[Figure~4]{typefree}.)
Several of the maximal cones are labeled by the corresponding $c$-sortable element.
Figure~\ref{g2s-1} shows the intersection of $\DF_c$ with the opposite affine plane $V_{-1}^*=\set{x\in V^*:\br{x,\delta}=-1}$.
This is a representation of $-\F_{c^{-1}}$, and several maximal cones are labeled by the corresponding $c^{-1}$-sortable element.
Both pictures are from the point of view of a point in $\Tits(A)$ separated from the origin by $V_1^*$.

Figure~\ref{g2s0} shows the intersection of $\DF_c$ with $V_0^*=\partial\Tits(A)$, from the same point of view.
\begin{figure}
\begin{picture}(0,0)
\end{picture}
\scalebox{.65}{
\includegraphics{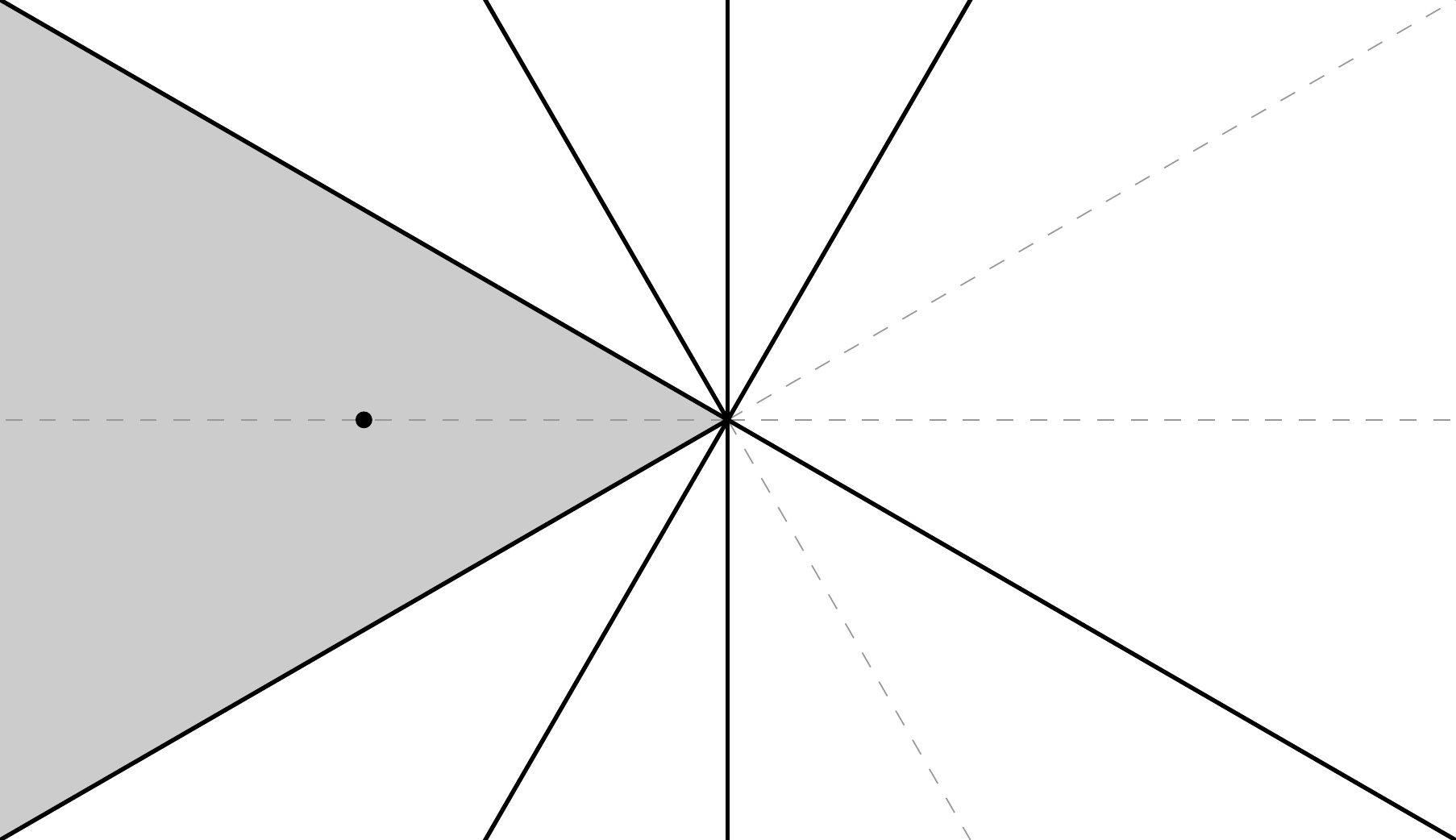}
\begin{picture}(0,0)(260,-150)
\put(-130,5){\LARGE$x_c$}
\put(5,60){\LARGE$s_3$}
\put(-51,100){\LARGE$s_1s_3s_1$}
\put(-175,110){\LARGE$s_1s_2s_3s_1s_2s_1s_2s_1$}
\put(-130,-90){\LARGE$s_1s_2s_1s_2s_1$}
\put(-81,-146){\LARGE$s_1s_2s_1s_2s_1s_2$}
\put(40,-60){\LARGE$s_2$}
\put(50,40){\LARGE$s_2s_3$}
\put(80,5){\LARGE\textcolor{red}{$s_1$}}
\put(-40,-95){\LARGE\textcolor{red}{$s_3s_1$}}
\put(30,-100){\LARGE\textcolor{red}{$s_3s_1s_3$}}
\put(-170,-116){\LARGE\textcolor{red}{$s_3s_2s_1s_2s_1s_2s_1$}}
\put(-105,70){\LARGE\textcolor{red}{$s_2s_1s_2s_1$}}
\put(-69,124){\LARGE\textcolor{red}{$s_2s_1s_2s_1s_2$}}
\put(0,142){\LARGE\textcolor{red}{$s_2s_1s_2s_1s_2s_1$}}
\end{picture}
}
\caption{$|\DF_c|\cap\partial\Tits(A)$ for Example~\ref{affineG2 slice ex}}
\label{g2s0}
\end{figure}
The shaded region indicates $V^*\setminus|\DF_c|$, and the vector $x_c$ is indicated.
Black lines indicate the boundaries of the intersections $F\cap\partial\Tits(A)$ for maximal cones $F$ of $\DF_c$.
Each maximal cone $F$ of $\DF_c$ such that $F\cap\partial\Tits(A)$ is \mbox{$2$-dimensional} is labeled in black by the $c$-sortable element $v$ such that $F=\Cone_c(v)$ and in red by the $c^{-1}$-sortable element $v'$ such that $F=-\Cone_{c^{-1}}(v')$.
Gray dashed lines indicate the decomposition of cones $F\cap\partial\Tits(A)$ into cones in the $W_0$-Coxeter fan.
A gray dashed line also indicates the decomposition of the cone $V^*\setminus|\DF_c|$ into two cones in the $W_0$-Coxeter fan.  
\end{example}

We now prepare to prove Theorem~\ref{TitsCapBdry} by establishing a sequence of lemmas.
The first is \cite[Lemma~3.8]{typefree} when $x$ and $y$ are roots, and holds for all $x$ and $y$ by linearity.
\begin{lemma} \label{OmegaInvarianceUseful}
If~$s$ is initial or final in $c$, then $\omega_c(x,y)=\omega_{scs}(sx,sy)$ for all vectors $x$ and $y$ in $V$.
\end{lemma}

\begin{lemma} \label{EasyCase} 
If $v$ is $c$-sortable, then there exists $\beta\in C_c(v)$ such that $\pi(\beta)$ is a positive scaling of a root in $\Phi_0^{\omega +}$. 
\end{lemma}

\begin{proof}
Since every root $\beta$ has $\pi(\beta)=\beta - r \delta$ for some $r$, we need to show that there exists $\beta\in C_c(v)$ such that $\omega_c(\beta,\delta)>0$.
We do so by induction on $\ell(v)$. 
Let $s$ be initial in $c$.

If $v \not \geq s$ then $\alpha_s\in C_c(v)$ and $\omega_c(\alpha_s,\alpha_{s'})\ge0$ for all $s'\in S$, with strict inequality for at least one $s'$.  
Since $\delta$ has strictly positive simple-root coordinates, we have $\omega_c(\alpha_s,\delta)>0$.

If $v \geq s$, then by induction there exists $\beta\in C_{scs}(sv)$ with $\omega_{scs}(\beta, \delta)>0$.
Equation \eqref{C recur} says $s\beta\in C_c(v)$ and Lemma~\ref{OmegaInvarianceUseful} implies that $\omega_{c}(s\beta, s\delta) >0$. 
The element $\delta$ is fixed by every element of $W$, so $\omega_{c}(s\beta, \delta) >0$.
\end{proof}

\begin{lemma}\label{Phi+Nonempty}
The sets $\Phi_0^{\omega +}$ and $\Phi_0^{\omega -}$ are nonempty.
\end{lemma}

\begin{proof}
Let $s$ be initial in $c$. Since $\delta$ is a strictly positive combination of the simple roots, we have $\omega_c(\alpha_s,\delta)>0$.
If $s\neq s_\aff$, then $\alpha_s \in \Phi^{\omega +}_0$ and $- \alpha_s \in \Phi^{\omega -}_0$.
If $s=s_\aff$, then take $s'$ final in $c$, so that $\omega_c(\alpha_{s'},\delta)<0$.
Since $s_\aff$ cannot also be final in $c$, we have $\alpha_{s'} \in \Phi^{\omega -}_0$ and $- \alpha_{s'} \in \Phi^{\omega +}_0$.
\end{proof} 

It is less obvious whether $\Phi^{\omega 0}_0$ is empty. As a consequence of the following lemma, $\Phi^{\omega 0}_0$ is nonempty except when $\Phi$ has rank $2$. That is to say, except when the Cartan matrix is $\left[ \begin{smallmatrix} 2 & -2 \\ -2 & 2 \end{smallmatrix} \right]$, $\left[ \begin{smallmatrix} 2 & -4 \\ -1 & 2 \end{smallmatrix} \right]$, or $\left[ \begin{smallmatrix} 2 & -1 \\ -4 & 2 \end{smallmatrix} \right]$.

\begin{lemma} \label{Phi0Rank}
The real span of $\Phi_0^{\omega 0}$ is the hyperplane $\{ x \in V_0 : \omega_c(x, \delta) =0 \}$ in~$V_0$.
\end{lemma}

In order to prove Lemma~\ref{Phi0Rank}, we recall some results on writing elements of $W_0$ as products of reflections. 
The \newword{reflection length} $\ell_T$ of an element $w\in W_0$ is the length of a minimal-length expression for $w$ as a product of reflections.
An expression $t_1\cdots t_k$ for $w$ as a product of reflections has minimal length if and only if the roots $\beta_{t_i}$ are linearly independent.
(This is a result of \cite{Carter};  see also \cite[Lemma~2.4.5]{Armstrong}.)
When the roots $\beta_{t_i}$ are linearly independent, their real span is the image of the map $w-\mathrm{Id}$ from $V_0$ to itself.
(This follows from results of \cite{BWorth,BWKpi1} and is established explicitly in the proof of \cite[Theorem~2.4.7]{Armstrong}.)
In particular, $\ell_T(w)$ is the dimension of the image of $w-\mathrm{Id}$.
Finally, if $t$ is a reflection, then $\ell_T(tw)=\ell_T(w)\pm1$; this follows from the well-known fact that there is a homomorphism from $W$ to $\set{\pm1}$ sending each reflection to $-1$.

\begin{proof}[Proof of Lemma~\ref{Phi0Rank}]
Define $U_c$ to be $\{ x \in V_0 : \omega_c(x, \delta) =0 \}$.
By Lemma~\ref{Phi+Nonempty}, $U_c$ is not all of $V_0$, so it is a hyperplane in $V_0$. 
The claim is that this hyperplane is spanned by the roots it contains.

By Lemma~\ref{OmegaInvarianceUseful} and since $\delta$ is fixed by the action of $W$, if $s$ is initial in $c$, we see that $U_{scs}=sU_c$.
Since the action of $s$ preserves the set of roots, the lemma for $c$ is equivalent to the lemma for $scs$.  
Accordingly, we may as well assume that $c=s_1s_2\cdots s_{n-1}s_{\aff}$, where  $s_1, \ldots, s_{n-1}$ are the elements of $S_0$ and $s_\aff$ is the remaining simple reflection.
Write $c_0$ for the Coxeter element $s_1s_2\cdots s_{n-1}$ of $W_0$.
We have $\ell_T(c_0)=n-1$ because the vectors $\alpha_s$ for $s\in S_0$ are independent.

Recall that $t_\theta$ is the reflection with respect to the root $\theta\in\Phi_0$ such that $\alpha_\aff$ is a positive scaling of $\delta-\theta$.
Since $t_\theta$ is a reflection, the element $c_0t_\theta$ has reflection length $\ell_T(c_0)\pm1=(n-1) \pm 1$,
but $n-1$ is the maximum reflection length of an element in $W_0$, so $\ell_T(c_0 t_{\theta}) = n-2$.
Thus $\mathrm{Im}(c_0t_\theta - \mathrm{Id})$ in $V_0$ is a hyperplane in $V_0$ which is spanned by $n-2$ roots of $\Phi_0$; call this hyperplane $U$.
We will show that $U=U_c$, so $U_c$ is also spanned by roots of $\Phi_0$.

Let $x \in U$, so that $x = c_0t_\theta y-y$ for some $y \in V_0$. 
Then using \eqref{simplify reflection}, we compute
\[cy=c_0s_{\aff}y=c_0(t_\theta y+\br{\theta\ck,y}\delta)=c_0t_\theta y+\br{\theta\ck,y}\delta=x+y+\br{\theta\ck,y}\delta.\]
Thus $x=cy-y-\br{\theta\ck,y}\delta$, and therefore
\[ \omega_c(x,\delta)=\omega_c(cy,\delta)-\omega_c(y,\delta) . \]
By $n$ applications of Lemma~\ref{OmegaInvarianceUseful}, we see that $\omega_c(cy, \delta)=\omega_c(y,c^{-1}\delta)$.
But $\delta$ is fixed by the action of $W$, so $\omega_c(cy,\delta)=\omega_c(y,\delta)$ and thus $\omega_c(x,\delta) =0$.
We have shown that $U\subseteq U_c$.
But $U$ and $U_c$ are both hyperplanes in $V_0$, so $U = U_c$ and we are done.
\end{proof}

We pause to point out an important consequence of Lemma~\ref{Phi0Rank}.
Recall from Section~\ref{faces boundary sec} the definition of the Coxeter fan of $W_0$ in $V_0^*$.

\begin{proposition}\label{detail}
The point $x_c$ is contained in the relative interior of a ray of the Coxeter fan of $W_0$.
\end{proposition}

\begin{proof}
Lemma~\ref{Phi0Rank} says that there are $n-2$ linearly independent roots $\beta_1,\ldots,\beta_{n-2}$ in $\Phi_0^{\omega0}$.  
The intersection $V^*_0\cap\bigcap_{i=1}^{n-2}\beta_i^\perp$ is a line in $V^*_0$ and since it is an intersection of reflecting hyperplanes for $W_0$, it is the union of two rays of the Coxeter fan. 
The point $x_c\in V^*_0$ is contained in this line because $\br{x_c,\beta_i}=-\omega_c(\beta_i,\delta)=0$ for all~$i$.  
Since $x_c\neq 0$ by Lemma~\ref{Phi+Nonempty}, it is contained in the relative interior of a ray of the Coxeter fan.
\end{proof}

The following lemma will assist us in applying Theorem~\ref{sortable is aligned} to the proof of Theorem~\ref{TitsCapBdry}.

\begin{lemma} \label{InABox}
Let $(\beta_1, \ldots, \beta_{n-1})$ be linearly independent roots in $\Phi_0$.
Then there are only finitely many elements $w$ of $W$ such that, for each $i$, the size of $\inv(w) \cap \Span_{\RR}(\beta_i, \delta)$ is at most $1$.
\end{lemma}

\begin{proof} 
Given a root $\beta\in\Phi_0$, if a positive scaling of $\beta+k\delta$ is a root for some integer~$k$, then we write $\beta^{(k)}$ for this root.
The set $\Phi \cap \Span_{\RR}(\beta, \delta)$ consists of the roots of the form $\beta^{(k)}$ and $(- \beta)^{(k)}$.
For every root $\beta$ in $\Phi_0$, there are infinitely many nonnegative integers $k$ for which there is a root of the form $\beta^{(k)}$.

Recall the notations $V_1^* := \{ x \in V^* : \langle x, \delta \rangle =1 \}$ and $D_1 := D \cap V_1^*$.
A positive root $\beta^{(k)}$ is an inversion of $w\in W$ if and only if $D_1$ and $wD_1$ are on opposite sides of the hyperplane $\{ x: \br{x,\beta} = -k \}$. 
Similarly, a positive root $(-\beta)^{(k)}$ is an inversion of $w$ if and only if $D_1$ and $wD_1$ are on opposite sides of the hyperplane $\{ x: \br{x,\beta} = k \}$.
Thus the condition that $\inv(w) \cap \Span_{\RR}(\beta_i, \delta)$ is at most~$1$ restricts $wD_1$ to lie between the two hyperplanes $\{ x: \br{x,\beta_i} = \sgn(\beta_i) \}$ and $\{ x: \br{x,\beta_i} =  -2 \sgn(\beta_i) \}$.

Since the $\beta_i$ are a basis for $V_0$, imposing that $w D_1$ lies between the $n-1$ pairs of parallel hyperplanes normal to the $\beta_i$ is equivalent to imposing that $w D_1$ lies in a parallelepiped.
Since $W$ acts by Euclidean motions on $V_1^*$, each of the simplices $w D_1$ is congruent.
But the parallelepiped is bounded and so has finite volume, and thus contains only finitely many simplices $w D_1$.
\end{proof}

We can now prove that the doubled Cambrian fan has the promised support.

\begin{proof}[Proof of Theorem~\ref{TitsCapBdry}]
We first show that $|\F_c| \cap \partial \Tits(A)$ is contained in or equal to
$\bigcup_{\beta \in \Phi_0^{\omega +}} \{ x \in V_0^* : \br{x,\beta} \geq 0 \}$. Since $|\F_c| = \bigcup_v \Cone_c(v)$, it is enough to check that $\Cone_c(v)\cap \partial \Tits(A)$ is contained in the right hand side for any $c$-sortable $v$. 
Fix $v$, and let $\beta$ be as in Lemma~\ref{EasyCase}. Let $\beta_0$ be the root in $\Phi_0^{\omega+}$ that is a positive 
scaling of $\pi(\beta)$.
Then $\Cone_c(v) \subseteq \{ x\in V^* : \br{x,\beta} \geq 0 \}$ and thus 
$\Cone_c(v) \cap \partial \Tits(A) \subseteq  \{ x \in V_0^* : \br{x,\beta_0} \geq 0 \}$.

The rest of the proof establishes the reverse containment.
Let $\beta\in\Phi_0^{\omega +}$; we must show that  ${\set{x\in V_0^* : \br{x,\beta} \geq 0}}$ is contained in $|\F_c|$. 
The set $\set{x \in V_0^* : \br{x,\beta}\geq 0}$ is a union of cones of $u_0 D_0$, for $u_0$ in $W_0$, where $D_0$ is $\bigcap_{s \in S_0} \set{x\in V_0^*: \br{x,\alpha_s}\ge 0}$.
Thus it is enough to show that, if $u_0 D_0 \subseteq \{ x \in V_0^* : \br{x,\beta} \geq 0 \}$, then $u_0 D_0 \subseteq |\F_c|$.

Lemma~\ref{Phi0Rank} implies (see, for example,~\cite{DyerAffine}) that $\Phi^{\omega0}_0$ is a root system of rank $n-2$, so there exist linearly independent elements $\beta_1, \ldots, \beta_{n-2}$ of $\Phi_0^{\omega0}$ such that $u_0D_0$ is contained in $\set{x\in V_0^*:\br{x,\beta_i}\ge0,\ 1 \leq i \leq n-2}$.
Defining $\beta_{n-1}$ to be $\beta$ (the given root in $\Phi_0^{\omega +}$) we have
$u_0D_0\subseteq\set{x\in V_0^*:\br{x,\beta_i}\ge0,\ 1 \leq i \leq n-1}$.  
Note that the roots $\beta_1, \ldots, \beta_{n-1}$ are linearly independent because $\beta\in\Phi_0^{\omega +}$.

We retain the notation $\beta^{(k)}$ from the proof of Lemma~\ref{InABox}.  
For each $\beta_i$, there is a smallest integer $k_i$ such that $(-\beta_i)^{(k_i)}$ is a positive root.
Define 
\[
\Lambda :=  \set{x\in V^{\ast} :\br{x,\delta} \geq 0} \cap  \bigcap_{i=1}^{n-1} \set{x\in V^{\ast} :\br{x,(- \beta_i)^{(k_i)}} \leq 0} . \]
Note that $(-\beta_i)^{(k_i)}$ is a linear combination of $\delta$ and a nonzero multiple of $\beta_i$, and that the vectors $\beta_i$ are linearly independent. So the $n$ normal vectors defining $\Lambda$ are linearly independent and $\Lambda$ is a simplicial cone, one of whose faces contains $u_0 D_0$.
Since $\Lambda$ lies on the $\Tits(A)$ side of $\delta^{\perp}$, the interior of $\Lambda$  is contained in $\Tits(A)$ and thus  is covered by its intersections with cones of $\F_c$. We claim that only finitely many cones of $\F_c$ meet the 
interior of $\Lambda$.

We pause to illustrate the claim by returning to Example~\ref{affineG2 slice ex}.
Let $\beta=2\alpha_1+3\alpha_2$ and let $u_0D_0$ be the cone spanned by $\rho_2$ and $-\rho_1+3\rho_2$.
This is the cone that is marked by (black) $s_1s_3s_1$ and (red) $s_2s_1s_2s_1s_2$ in Figure~\ref{g2s0}.
Because $\Phi_0^{\omega0}=\set{\pm(\alpha_1+2\alpha_2)}$, we must take $\beta_1=\alpha_1+2\alpha_2$ and by construction $\beta_2$ is $\beta=2\alpha_1+3\alpha_2$.
We have shaded $\Lambda \cap V_1^{\ast}$ in Figure~\ref{g2s1small}, a miniature version of Figure~\ref{g2s1}.
The claim is that only finitely many cones $\Cone_c(v)$ meet $\Lambda$.
In this case, we see that 14 cones $\Cone_c(v)$ in $\F_c$ meet $\Lambda$. 
\begin{figure}[ht]
\scalebox{1.5}{
\includegraphics{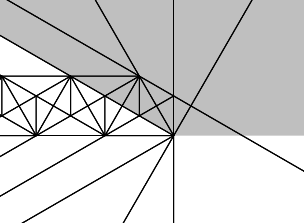}
}
\caption{The shaded region is the set $\Lambda \cap V_1^{\ast}$ used in the proof of Theorem~\ref{TitsCapBdry}. Compare to Figure~\ref{g2s1}}
\label{g2s1small}
\end{figure}

To prove the claim, suppose that $v\in W$ is $c$-sortable and that $\Cone_c(v)$ meets the 
interior of $\Lambda$. 
Then there exists $w\in W$ with $wD \subseteq \Lambda$ and $\pidown^c(w) =v$. 
Since $wD$ is contained in $\set{x\in\Tits(A):\br{x,(-\beta_i)^{(k_i)}}\le 0}$, we know that $wD$ is \emph{not} contained in $\set{x\in\Tits(A):\br{x,(\beta_i)^{(k)}}\leq 0}$ 
for any positive root $\beta_i^{(k)}$.
So such a $\beta_i^{(k)}$ is not an inversion of $w$ and, since $w \geq \pidown^c(w) = v$ (Theorem~\ref{pidown fibers}), no positive root $\beta_i^{(k)}$ 
can be an inversion of $v$.

But also, $v$ is $c$-sortable and $\omega_c(\beta_i, \delta) \geq 0$.  
Let $\Phi'$ be the generalized rank-two parabolic sub root system $\Phi\cap\Span_\RR(\beta_i,\delta)$.
Theorem~\ref{sortable is aligned} says in particular that $v$ is $c$-aligned with respect to $\Phi'$, and therefore at most one positive root of the form $(- \beta_i)^{(k)}$ can be an inversion of $v$ for each $i$.
We have shown that, if $\Cone_c(v)$ meets the interior of $\Lambda$, then (for each $i$) there is at most one inversion of  $v$ of the form $(\pm \beta_i)^{(k)}$.
Now Lemma~\ref{InABox} proves the claim.

Let $\mathcal{V}$ be $\bigcup \Cone_c(v)$, where the union ranges over those $v$ such that $\Cone_c(v)$ meets the relative interior of $\Lambda$.
We have just shown that this union is finite, so $\mathcal{V}$ is a finite union of closed cones and is itself closed.   
By Theorem~\ref{pidown fibers}, $\mathcal{V}$ contains the interior of $\Lambda$, and since $\mathcal{V}$ is closed, it also contains the facet of $\Lambda$ containing $u_0D_0$.  
We have shown that $u_0 D_0 \subseteq |\F_c|$, as desired.

We have shown that $|\F_c|\cap\partial\Tits(A)$ equals $\bigcup_{\beta \in \Phi_0^{\omega +}} \set{x\in V_0^* : \br{x,\beta} \geq 0}$, which is nonempty by Lemma~\ref{Phi+Nonempty}.
An analogous argument shows that $|\!-\!\F_{c^{-1}}|\cap \partial \Tits(A)$ equals $\bigcup_{\beta \in \Phi_0^{\omega +}} \{ x \in V_0^* : \br{x,\beta} \geq 0 \}$.
(The negative signs introduced by the antipodal map cancel the negative signs introduced by passing from $c$ to $c^{-1}$.)
Since $|\F_c|\cap\partial\Tits(A)$ and $|\!-\!\F_{c^{-1}}|\cap \partial \Tits(A)$ coincide, they both equal $|\DF_c|\cap\partial\Tits(A)$, which therefore equals $\bigcup_{\beta \in \Phi_0^{\omega +}} \set{x\in V_0^* : \br{x,\beta} \geq 0}$.
\end{proof}

\begin{proof}[Proof of Corollary~\ref{DFc complement}]
By Theorem~\ref{pidown fibers},
the interiors of $\Tits(A)$ and $- \Tits(A)$ lie in $|\F_c|$ and $-|\F_{c^{-1}}|$. So $V^* \setminus |\DF_c| = \partial \Tits(A) \setminus \left(|\DF_c|\cap \partial \Tits(A)\right)$.
 The first description of $V^* \setminus |\DF_c|$ follows by taking the complement of the expression for $|\DF_c|\cap \partial \Tits(A)$ in Theorem~\ref{TitsCapBdry}.

We now prove the second description. Every cone of $\DF_c$
intersects $\partial \Tits(A)$ in a union of cones of the $W_0$-Coxeter fan.  
We must show that a cone $F$ of the $W_0$-Coxeter fan is contained in $|\DF_c|$ if and only if $x_c\not\in F$.

For each $\beta \in \Phi_0$, the cone $F$ lies 
entirely to one side of $\beta^{\perp}$ or the other, or on $\beta^{\perp}$.  
Thus we deduce from Theorem~\ref{TitsCapBdry} that $F \subseteq |\DF_c|$ if and only if  there exists $\beta \in \Phi_0^{\omega +}$ such that $F \subseteq \set{x\in V_0^*:\br{x,\beta}\geq 0}$.
But also $\beta\in\Phi^{\omega+}_0$ if and only if $\omega_c(\beta,\delta)>0$, or equivalently $\br{x_c,\beta}<0$.
So $F \subseteq |\DF_c|$ if and only if  there exists $\beta\in\Phi_0$ with both (a) $F \subseteq \set{x\in V_0^*:\br{x,\beta}\geq 0}$ and (b) $\br{x_c,\beta}<0$. 
If both (a) and (b) hold for some $\beta$, then necessarily $x_c\not\in F$, and conversely, if $x_c\not\in F$ then some facet-defining hyperplane of $F$ 
separates $x_c$ from $F$.
This hyperplane is $\beta^\perp$ for some $\beta\in\Phi_0$.
Therefore either $\beta$ or $-\beta$ satisfies (a) and (b).

The first description writes $V^* \setminus |\DF_c|$ as a finite intersection of open cones in $\partial\Tits(A)$, so $V^* \setminus |\DF_c|$ is an open cone in $\partial\Tits(A)$.
The second description implies that $V^* \setminus |\DF_c|$ is nonempty.
\end{proof}

\subsection{Well-connectedness}\label{affine well-conn sec} 
In this section, we prove that $(\DCamb_c,\DC_c)$ is well-connected and prove a corollary to well-connectedness which is helpful in proving simple connectivity.

\begin{prop}\label{well-connected}
If $A$ is of affine type, then $(\DCamb_c,\DC_c)$ is well-connected.
\end{prop}
\begin{proof}
The polyhedral property was already established.
It remains to show that if $F$ is a proper face of two maximal cones in $\DF_c$, then there exists a path in $\DCamb_c$ connecting the two corresponding vertices, such that, along the path, each corresponding cone has $F$ as a face.
We already know that $(\Camb_c,C_c)$ and $(\Camb_{c^{-1}},C_{c^{-1}})$ have this property by Theorem~\ref{camb good}. 
By Corollary~\ref{DoubledConeExists}, we need only check the case where $v$ is $c$-sortable, $v'$ is $c^{-1}$-sortable, the interior of $\Cone_c(v)$ is in $\Tits(A)$ and the interior of $-\Cone_{c^{-1}}(v')$ is in $-\Tits(A)$.
In particular, $F$ is in the boundary of $\Tits(A)$.

In the first paragraph of the proof of Theorem~\ref{TitsCapBdry}, we showed that there exists $\beta\in\Phi_0^{\omega+}$ such that $\Cone_c(v) \cap \partial \Tits(A) \subseteq  \{ x\in V_0^* : \br{x,\beta} \geq 0 \}$.
Since $F$ is a face of $\Cone_c(v)$ and is contained in $\partial\Tits(A)=V^*_0$, we have $F\subseteq\set{x\in V_0^*:\br{x,\beta}\geq 0}$ for the same $\beta\in\Phi_0^{\omega+}$.
Choose $x$ in the relative interior of~$F$.
By Proposition~\ref{face in boundary}, $F$ has dimension less than $n-1$ as do all other faces of $\F_c$ contained in $\partial\Tits(A)$.
Thus there exists a vector $y\in V_0^*$ such that, for small enough $\ep>0$, the vector $x+\ep y$ is not contained in any non-maximal face of $\F_c$.
Furthermore, we can choose $y$ such that $\br{y,\beta}\ge0$, so that $x+\ep y\in|\F_c|$ for all $\ep>0$ by Theorem~\ref{TitsCapBdry}.
We conclude that for small enough $\ep$, the point $x+\ep y$ is in the interior of some $\Cone_c(u)$ containing $F$ in its boundary.

Since $\F_c$ is a fan and $F$ is a face of $\Cone_c(v)$, we see that $F$ is a face of $\Cone_c(u)$.
Since $(\Camb_c,C_c)$ is well-connected, there is a path in $\Camb_c$ connecting $u$ to $v$ such that, along the path, each corresponding cone has $F$ as a face.  
Furthermore, by Corollary~\ref{DoubledConeExists}, there exists a $c^{-1}$-sortable element $u'$ such that $\Cone_c(u)=-\Cone_{c^{-1}}(u')$.
Thus there is a path in $-\Camb_{c^{-1}}$ from $u'$ to $v'$ such that, along the path, each corresponding cone has $F$ as a face.
Concatenating the two paths, we obtain the desired path in $\DCamb_c$.
\end{proof}

\begin{cor} \label{finite type interior}  
Suppose $A$ is of affine type and let $F$ be a face of $\DF_c$ of dimension $n-2$. 
Then the maximal cones of $\DF_c$ containing $F$ form a rank-two cycle or path in $\DCamb_c$. 
If the relative interior of $F$ is in the interior of $|\DF_c|$, then they form a cycle. 
Otherwise they form a path.
Moreover, in the former case, the rank two root system $\Phi' := \{ \beta \in \Phi : \langle \beta, F \rangle = 0 \}$ is finite and, in the latter case, $\Phi'$ is infinite of affine type.
\end{cor}

\begin{proof}
The first assertion of the corollary
is immediate by Propositions~\ref{rk2cyc poly} and \ref{well-connected}.
Let $\tau$ be the rank-two cycle or path in $\DCamb_c$ associated to~$F$.
We write $\relint(F)$ for the relative interior of $F$ and break into cases based on the relationship between $\relint(F)$ and $\partial\Tits(A)$.  

If the relative interior of $F$ intersects $\Tits(A)$, then let $x\in\relint(F)\cap\Tits(A)$.
The stabilizer of $x$ with respect to the action of $W$ on $V^*$ is a proper (not-necessarily standard) parabolic subgroup of $W$ generated by reflections in the roots $\Phi''={ \{ \beta \in \Phi : \langle \beta, x \rangle = 0 \}}$.  
(This is a standard fact.  Since $\Tits(A)$ is the union of cones $uD$ for $u\in W$, there exists $w\in W$ such that $wx\in D$.
The stabilizer of $wx$ is then described by \cite[Theorem~5.13]{Humphreys}.)
Since $W$ is of affine type, its proper parabolic subgroups are all finite, and we conclude that $\Phi''$ is finite.
The root system $\Phi'$ described in the statement of the corollary is a sub root system of $\Phi''$, and thus $\Phi'$ is also finite.
Proposition~\ref{rk2cyc reflection} now implies that $\tau$ is a cycle.
Therefore the relative interior of $F$ lies in the interior of $|\DF_c|$, so the proposition holds in this case.

If the relative interior of $F$ intersects $-\Tits(A)$, then the analogous argument works.
Thus it remains to consider the case where $F$ is contained in $\partial\Tits(A)$.
In this case, $\delta\in\Span_\RR(\Phi')$,  
We conclude that $\Phi'$ is infinite, and since $\Phi'$ is a sub root system of $\Phi$, it is necessarily of affine type.
Proposition~\ref{rk2cyc reflection} now implies that $\tau$ is a path.
To complete the proof, we must show that $F$ is not in the interior of $|\DF_c|$.
Supposing to the contrary, if we view $F$ as a cone in the fan $\DF_c\cap\partial\Tits(A)$, the cone $F$ is shared by two $(n-1)$-dimensional faces $G_1$ and $G_2$ of $\DF_c\cap\partial\Tits(A)$.
Proposition~\ref{face in boundary} says that each $G_i$ is the intersection with $\partial\Tits(A)$ of a maximal cone $C_i$ of $\DF_c$.
Every other maximal cone of $\DF_c$ containing $F$ has two facets defined by roots in~$\Phi'$.
But then, since the hyperplanes $\beta^\perp$ for $\beta\in\Phi'$ accumulate to $\delta^\perp=\partial\Tits(A)$, there are only finitely many different maximal cones in $\DF_c$ containing $F$ as a face, contradicting the fact that $\tau$ is an infinite path..
The situation is illustrated schematically in Figure~\ref{schematic}.
In the picture, the bold lines indicate facets of $C_1$ and $C_2$ and the lighter lines indicate hyperplanes $\beta^\perp$ for $\beta\in\Phi'$.
\end{proof}
\begin{figure}[ht]
\scalebox{0.75}{
\begin{picture}(0,0)(80,-40)
\end{picture}
\includegraphics{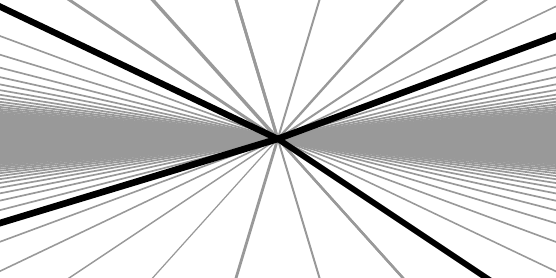}
\begin{picture}(0,0)(80,-40)
\put(-50,-3){\Large$C_2$}
\put(36,-3){\Large$C_1$}
\put(-9,-13){\Large$F$}
\end{picture}
}
\caption{An illustration for the proof of Corollary~\ref{finite type interior}}
\label{schematic}
\end{figure}

The proof of Corollary~\ref{finite type interior} also establishes the following useful fact, which augments Proposition~\ref{face in boundary}.

\begin{cor}\label{2face in boundary}
If $A$ is of affine type, then every $(n-2)$-dimensional face of $\DF_c$ contained in $\partial\Tits(A)$ is in the boundary of $|\DF_c|$.
\end{cor}

\subsection{Simple connectivity}\label{affine simply sec}\label{debt paid sec}
The final piece in the proof of Theorem~\ref{AffineDoubleFramework} is the following proposition.
\begin{prop}\label{simply}
If $A$ is of affine type, then $(\DCamb_c,\DC_c)$ is simply connected.  
\end{prop}

In preparation for the proof of the proposition, consider a loop $v_0 \to v_1 \to v_2 \to \cdots \to v_p = v_0$ in $\DCamb_c$.  
Choose generic points $y_i$ in the interior of the cones $\Cone_c(v_i)$ so that the line segment $\overline{y_i y_{i+1}}$ crosses between $\Cone_c(v_i)$ and $\Cone_c(v_{i+1})$ in the relative interior of the wall $\Cone_c(v_i) \cap \Cone_c(v_{i+1})$.
In particular, $\overline{y_i y_{i+1}}$ is contained entirely in the interior of $|\DF_c|$. We write $\gamma$ for the loop $\bigcup_{i=0}^{p-1} \overline{y_i y_{i+1}}$ in $V^{\ast}$.  

Let $\sigma$ be the union of the closed cones of the $W_0$-Coxeter fan that contain $x_c$.
From Corollary~\ref{DFc complement}, $V^* \setminus \sigma$ is the interior of $|\DF_c|$.
Choose $z$ to be a generic point in the relative interior of $- \sigma$. 
For each $i$ from $0$ to $p-1$, write $\Delta_i$ for the triangle $\mathrm{ConvexHull}(z, y_i, y_{i+1})$.

\begin{lemma}\label{Convex Lemma}
The cone $\sigma$ does not intersect $\Delta_i$.
\end{lemma}

\begin{proof}
If the line segment $\overline{y_i y_{i+1}}$ does not cross $\partial \Tits(A)$, then the claim is obvious, since $z$ is the only point of $\Delta_i$ in $\partial \Tits(A)$ while $\sigma$ is contained in $\partial \Tits(A)$ and we constructed $z$ to 
lie outside $\sigma$.

If $\overline{y_i y_{i+1}}$ crosses $\partial \Tits(A)$, let $w$ be the point $\overline{y_i y_{i+1}} \cap \partial \Tits(A)$. Then $\Delta_i \cap \partial \Tits(A)$ is the line segment $\overline{wz}$. Now $\overline{y_i y_{i+1}}$ lies in 
the interior of $|\DF_c|$, so $w$ is in the relative interior of $|\DF_c|\cap\partial \Tits(A)$.
By Theorem~\ref{TitsCapBdry}, there is some $\beta$ in $\Phi_0^{\omega +}$ so that $\br{w,\beta}>0$. 
But $z$ is in the relative interior of $-\sigma$, which equals $\bigcap_{\beta \in \Phi_0^{\omega +}} \{ x \in \partial \Tits(A) : \br{x,\beta} > 0 \}$, so also $\br{z,\beta}>0$.  \
Therefore the entire segment $\overline{wz}$ lies strictly 
to one side of $\beta^{\perp}$, and thus misses $\sigma$.
\end{proof}

Assuming we have chosen our points $y_i$ and $z$ generically enough, $\Delta_i$ only meets the faces of dimension $\geq n-2$ in $\DF_c$, and meets them transversely,  $\partial \Delta_i$ only meets the faces of dimension $\ge n-1$, and $z$ lies in a maximal cone of $\DF_c$.

\begin{lemma} \label{Finite Geometry Lemma}
Only finitely many cones of $\DF_c$ meet $\Delta_i$.
\end{lemma}

\begin{proof}
Suppose otherwise.
Since $\Delta_i$ is transverse to $\DF_c$, if $\Delta_i$ meets a cone, it meets its relative interior. 
For each such cone
$C$, choose a point $q(C)$ in the intersection of $\Delta$ and the relative interior of $C$.
Since $\Delta_i$ is compact, the set of such $q(C)$ must have an accumulation point $r\in\Delta_i$. 
Let $F$ be the cone of $\DF_c$ in whose relative interior $r$ lies.

Since $\Delta_i$ is disjoint from all faces of $\DF_c$ of dimension less than $n-2$, we know that $\dim F\ge n-2$.
If $\dim F=n$, then we have a contradiction, as there is an open neighborhood of $r$ in $V^*$ which meets no cones of $\DF_c$ other than $F$.
If $\dim F=n-1$, we likewise have a contradiction, as there is an open neighborhood of $r$ in $V^*$ which meets only three cones of $\DF_c$ (namely, $F$ and the two cones containing it.
Finally, suppose that $\dim F = n-2$. 
Since $\Delta_i$ lies in the interior of $\DF_c$ (Lemma~\ref{Convex Lemma}), the fact that $\relint(F)$ meets $\Delta_i$ means that $\relint(F)$ lies in the interior of $\DF_c$.
From Corollary~\ref{finite type interior}, $F$ lies in finitely many cones of $\DF_c$ and we have a contradiction as in the other cases.
\end{proof}

\begin{proof}[Proof of Proposition~\ref{simply}] 
Continuing the notation from above, we want to show that the loop $v_0 \to v_1 \to v_2 \to \cdots \to v_p = v_0$ is trivial modulo relations coming from finite rank-two cycles in $\DCamb_c$. 
We chose $z$ to lie in $\Cone_c(w)$ for some vertex $w$ of $\DCamb_c$.
We write $v_i \leadsto w$ for the path from $v_i$ to $w$ in $\DCamb_c$ following the edge $y_i z$ of $\Delta_i$; the wavy line reminds us that we may pass through other cones of $\DF_c$ on the way.
We also write $w\leadsto v_i$ for the reverse path.

Pulling the path from $y_i$ to $y_{i+1}$ across the triangle $\Delta_i$, from the edge $\overline{y_i y_{i+1}}$ of the triangle to the two edges $\overline{y_i z} \cup \overline{z y_{i+1}}$, we cross over codimension-$2$ faces of $\DF_c$. 
By Lemma~\ref{Finite Geometry Lemma}, we only cross finitely many of them.
From Corollary~\ref{finite type interior}, all of these codimension-$2$ faces correspond to finite rank-two cycles in $\DCamb_c$. 
So $v_i \to v_{i+1}$ is equivalent to $v_i \leadsto w \leadsto v_{i+1}$ modulo rank-two cycles. 
Thus the cycle $v_0 \to v_1 \to v_2 \to \cdots \to v_p = v_0$ is equivalent to the following cycle, which is obviously trivial:  
$ v_0 \leadsto w \leadsto  v_1 \leadsto w \leadsto  v_2 \leadsto \cdots \leadsto v_{p-1} \leadsto w \leadsto  v_p = v_0$. 
\end{proof}

Having proved Theorem~\ref{AffineDoubleFramework}, we conclude by paying an old debt. 
Almost all of Corollary~\ref{affine type conj} follows from Theorems~\ref{AffineDoubleFramework} and~\ref{all conj}.
The remaining piece is Conjecture~3.14 of \cite{framework} (the sign coherence of $\g$-vectors), which is now easy now that we have the definition of the doubled Cambrian framework.

 \begin{proof}[End of the proof of Corollary~\ref{affine type conj}] 
Proposition~\ref{camb g sign-coherent} implies that the conclusion of Conjecture~3.14 of \cite{framework} holds in the Cambrian reflection frameworks $(\Camb_c,C_c)$ and $(\Camb_{c^{-1}},C_{c^{-1}})$.
The antipodal map preserves sign-coherence of $\g$-vectors, so Conjecture~3.14 of \cite{framework} holds for all vertices of $(\DCamb_c,\DC_c)$.
\end{proof}

\section*{Acknowledgments}
The authors thank Salvatore Stella for helpful conversations, leading in particular to a simplification of the proof of Lemma~\ref{Phi0Rank}.  
The authors also thank the anonymous referees for their very careful reading and many helpful suggestions.

\end{document}